\documentclass{amsart}

\usepackage{amssymb, mathtools, delimset}
\DeclarePairedDelimiter\ceil{\lceil}{\rceil}
\DeclarePairedDelimiter\floor{\lfloor}{\rfloor}

\usepackage{aliascnt}

\usepackage{array, booktabs, enumitem}
\setlist{itemsep=4pt, topsep=0pt, leftmargin=17pt}

\usepackage{xcolor}
\definecolor{BIT}{cmyk}{1, 0, 1, 0}
\usepackage[numbers]{natbib}
\usepackage[
pagebackref,
colorlinks=true,
linkcolor=blue,
citecolor=BIT,
urlcolor=blue,
linktoc=all
]{hyperref}
\usepackage[capitalize, nameinlink, noabbrev]{cleveref}

\crefname{def}{Def.}{Defs.}
\Crefname{def}{Def.}{Defs.}
\creflabelformat{def}{~\upshape(#2#1#3)}
\crefname{ineq}{Ineq.}{Ineqs.}
\Crefname{ineq}{Ineq.}{Ineqs.}
\creflabelformat{ineq}{~\upshape(#2#1#3)}
\makeatletter
\newcommand\creflabel[2][\@currentcounter]{%
\crefalias{\@currentcounter}{#1}\label{#2}}
\makeatother

\makeatletter
\let\NAT@orig@test\NAT@test
\renewcommand{\NAT@test}[1]{%
\@ifnum{#1=\@ne}{%
\@ifnum{\NAT@ctype=\z@}{%
\NAT@hyper@{\NAT@orig@test{#1}}%
}{\NAT@orig@test{#1}}%
}{\NAT@orig@test{#1}}%
}
\makeatother

\hypersetup{
pdftitle={Schur positivity of three-legged spiders},
pdfauthor={David G.L. Wang and Watson Z.Y. Wang},
pdfkeywords={chromatic symmetric function,
Schur positivity,
spider graph,
Pieri rule}
}

\allowdisplaybreaks[2]
\numberwithin{equation}{section}
\numberwithin{table}{section}

\newtheorem{theorem}{Theorem}[section]
\newaliascnt{lemma}{theorem}
\newtheorem{lemma}[lemma]{Lemma}
\aliascntresetthe{lemma}
\newaliascnt{proposition}{theorem}
\newtheorem{proposition}[proposition]{Proposition}
\aliascntresetthe{proposition}
\newaliascnt{corollary}{theorem}
\newtheorem{corollary}[corollary]{Corollary}
\aliascntresetthe{corollary}

\DeclareRobustCommand{\Coeff}[2]{\brk[s]1{#1}#2}
\newcommand{\gs}{\succeq_s}
\newcommand{\spos}{\gs 0}
\newcommand{\epos}{\succeq_e 0}

\crefname{table}{Tab.}{Tabs.}
\Crefname{table}{Tab.}{Tabs.}
\crefname{equation}{Eq.}{Eqs.}
\Crefname{equation}{Eq.}{Eqs.}
\crefformat{pieri}{#2Pieri's rule#3}
\Crefformat{pieri}{#2Pieri's rule#3}

\crefformat{section}{#2\S#1#3}
\Crefformat{section}{#2Section~#1#3}
\crefrangeformat{section}{#3\S#1#4\crefrangeconjunction#5\S#2#6}
\Crefrangeformat{section}{#3Sections~#1#4\crefrangeconjunction#5#2#6}
\crefmultiformat{section}{#2\S#1#3}
{\crefpairconjunction#2\S#1#3}
{\crefmiddleconjunction#2\S#1#3}
{\creflastconjunction#2\S#1#3}
\Crefmultiformat{section}{#2Sections~#1#3}
{\crefpairconjunction#2#1#3}
{\crefmiddleconjunction#2#1#3}
{\creflastconjunction#2#1#3}
\crefalias{subsection}{section}

\newcommand{\creflastconjunction}{, and\nobreakspace}
\newcommand{\crefrangeconjunction}{--}

\title[Schur positivity of spiders]
{Schur positivity of three-legged spiders}

\author[D.G.L. Wang]{David G.L. Wang}
\address{School of Mathematics and Statistics,
Beijing Institute of Technology,
Beijing 102400,
China}
\email{glw@bit.edu.cn}
\thanks{The work of David G.L. Wang was supported
by the General Program of the National Natural Science Foundation of China
(Grant No.~12171034).}

\author[W.Z.Y. Wang]{Watson Z.Y. Wang\textsuperscript{*}}
\address{School of Mathematics and Statistics,
Beijing Institute of Technology,
Beijing 102400,
China}
\email{watsonzywang@bit.edu.cn}
\thanks{\textsuperscript{*}Corresponding author.}

\date{}

\begin{document}
\begin{abstract}
We give a complete classification of Schur positivity for three-legged spiders. Every spider with at least two even legs is Schur positive, whereas every spider with three odd legs has a negative Schur coefficient. If $a$ is even and $b,c$ are odd, then the spider $S(a,b,c)$ is Schur positive if and only if $a\le 5b+5c+2$. Whenever Schur positivity fails, there is exactly one negative Schur coefficient, whose value we determine explicitly. These results extend those of Thibon and Wang for the families $S(a,2,1)$ and $S(a,4,1)$, and those of Wang and Wang for $S(a,b,2)$. The proof combines known Schur-positivity results for clique-spiders with Pieri's rule and simultaneous induction.
\end{abstract}
\maketitle

\section{Introduction and main results}\label{sec:introduction}

\citet{Sta95} introduced the \emph{chromatic symmetric function}
of a finite simple graph~$G$ as
\[
X_G
=\sum_{\kappa\text{ proper}}\prod_{v\in V(G)}x_{\kappa(v)},
\]
where the sum is over all proper colorings of~$G$ by positive integers.
It refines the chromatic polynomial by recording the multiplicities
of the colors.
A graph is \emph{$e$-positive} if its chromatic symmetric function
has nonnegative coefficients in the elementary basis,
and \emph{Schur positive} if it has nonnegative coefficients in the Schur basis.
The \emph{Pieri rule} implies that~$e$-positivity is stronger
than Schur positivity.
For symmetric functions~$F$ and~$H$,
we write $F\gs H$ if~$F-H$ is Schur positive.
For any graph $G$,
we write $G\spos$ if $X_G\spos$.

Recent developments have changed the direction
of the positivity study in general.
\citet{Hik26} proved that every unit interval graph is~$e$-positive,
resolving the Stanley--Stembridge conjecture.
For Schur positivity, \citet{Sta98} recorded the conjecture,
attributed to Gasharov,
that every claw-free graph is Schur positive.
This was disproved independently by \citet{MM26X} and \citet{Pra26X},
who found the same pair of counterexamples.
These developments leave the classification of Schur-positive graphs
outside the unit interval class as a natural problem.

We study three-legged spiders,
the trees with a single branching vertex of degree~$3$.
For positive integers~$a,b,c$,
let~$S(a,b,c)$ be the tree obtained by identifying one endpoint
of each of three paths of lengths~$a,b,c$,
respectively.
Thus the leg lengths count edges and~$S(a,b,c)$ has $a+b+c+1$ vertices.
Permuting the leg lengths does not change its isomorphism type.
These graphs have the simplest branching structure
among trees that are not paths,
but they already exhibit both Schur positivity and its failure.

Spiders also play a role in the study of~$e$-positive trees.
\citet[Lem.~12]{DSv20} related connected partitions of a tree
to those of a spider whose leg lengths are the orders
of the components remaining after a vertex is deleted.
They conjectured that no tree of maximum degree at least~$4$ is~$e$-positive.
\citet{Zhe22} proved the conjecture for maximum degree at least~$6$,
and \citet{Tom26} extended it to maximum degree at least~$5$
and proved that every four-legged spider fails to be~$e$-positive.
The three-legged case therefore remains central
to the positivity problem for trees.
Schur positivity is a useful part of this problem
because it separates obstructions to both positivity properties
from the finer obstructions to~$e$-positivity alone.

Earlier results reveal a strong dependence of Schur positivity
on the parities of the leg lengths.
\citet[Thm.~2.3 and Eg.~2.4]{WW20}
proved that every spider with at least three odd legs fails to be Schur positive,
using the stable partition criterion of \citet[Prop.~1.5]{Sta98}.
\citet{WW23-DAM} classified the families $S(a,1,1)$ and $S(a,3,1)$;
in each family,
Schur positivity holds precisely for a finite interval of even values of $a$.
In contrast,
\citet[Thms.~4.2 and~4.6]{TW23X} proved Schur positivity
for $S(a,2,1)$ and $S(a,4,1)$,
and \citet[Thm.~9.9]{WW26X} proved it for every spider $S(a,b,2)$,
using the result for $S(a,2,1)$ as a base case.

Even these restricted families required substantial arguments.
Their $e$-positivity classifications were already known,
and \citet[\S1]{WZ25} explicitly noted
the difficulty of extending the Schur-positivity argument
from $S(a,2,1)$ to $S(a,4,1)$.

Our original aim was to extend the results of \citet{TW23X}
for~$S(a,2,1)$ and~$S(a,4,1)$ to broader families of three-legged spiders.
In the computed examples,
many positive Schur coefficients were large and appeared to grow rapidly
as the leg lengths increased.
These observations encouraged us to seek a general classification,
but until the present work we had no criterion covering arbitrary leg lengths.
AI-assisted exploration of proof strategies helped us develop the arguments
that yield the following complete classification;
see \cref{sec:ai-use} for the roles of the tools used.

\begin{theorem}[Complete classification]\label{thm:classification}
Let $G=S(a,b,c)$ be a spider,
where $a,b,c\ge 1$.
\begin{enumerate}[label=\textup{(\roman*)}]
\item If~$a,b,c$ are all odd,
then~$G$ is not Schur positive.
\item
If only one leg is even,
say, $a$ is even,
then $G\spos$ if and only if $a\le 5b+5c+2$.
\item
If at least two of~$a,b,c$ are even,
then $G\spos$.
\end{enumerate}
\end{theorem}

The classification has the following immediate consequence.

\begin{corollary}\label{cor:unit-main}
If $a\ge b\ge 1$ and~$b$ is even,
then~$X_{S(a,b,1)}\spos$.
\end{corollary}

\begin{proof}
For odd~$a$,
apply \cref{thm:classification}(ii) to~$S(b,a,1)$,
using $b\le a\le 5a+7$.
For even~$a$, apply \cref{thm:classification}(iii).
\end{proof}

The $b=2$ and $b=4$ cases in \cref{cor:unit-main}
were confirmed by \citet{TW23X} using noncommutative symmetric functions.
For $3$-legged spiders that are not Schur positive,
we determine the precise form of their failure of Schur positivity:
each of them has a unique negative Schur coefficient.

\begin{theorem}\label{thm:refined-bounds}
Let~$a,b,c\in \mathbb Z_{\ge 1}$ with $b,c$ odd,
and $h=\floor{(a+b+c+1)/2}$.
Suppose that the spider $G=S(a,b,c)$ is not Schur positive.
Then $X_G$ has a unique negative Schur coefficient.
Moreover,
either (i) $a$ is odd and $\Coeff{s_{hh}}{X_G}=-1$,
or (ii) $a$ is even and $\Coeff{s_{hh1}}{X_G}=(5b+5c+2-a)/2<0$.
\end{theorem}

This paper is organized as follows.
\Cref{sec:preliminaries}
collects the symmetric-function and graph identities used in the proofs.
\Cref{sec:master-recurrence} establishes the master recurrence,
clique-spider positivity,
and the resulting growth and propagation inequalities.
\Cref{sec:coefficient-bounds} determines the critical Schur coefficients
of spiders and proves Schur bounds for path differences.
\Cref{sec:proof-two-odd} establishes the bounds for spiders with two odd legs,
including the all-odd case and the one-even-leg case at or below the threshold.
\Cref{sec:proof-two-even} proves the remaining bounds
and completes the proofs of \cref{thm:classification,thm:refined-bounds}.

\section{Preliminaries}\label{sec:preliminaries}

We use~$\mathbb Z$ and~$\mathbb Q$ for the integers and rationals,
with subscripts indicating restrictions, as in~$\mathbb Z_{\ge 1}$.
A \emph{composition} of a positive integer~$n$
is a finite sequence $I=i_1\dotsm i_\ell$
of positive integers whose sum is~$n$,
written~$I\vDash n$.
It has \emph{size}~$|I|=n$ and \emph{length}~$\ell(I)=\ell$.
The empty sequence~$\emptyset$ is the unique composition of~$0$.
A \emph{partition} of~$n$ is a composition with weakly decreasing parts,
written $\lambda=\lambda_1\dotsm\lambda_\ell\vdash n$.
We use exponents to denote repeated parts,
as in~$r^21^2=(r,r,1,1)$.
For partitions~$\lambda,\mu$ of the same size,
we say that $\lambda$ is \emph{dominated} by $\mu$,
written as~$\lambda\unlhd \mu$,
if
\[
\sum_{j=1}^k\lambda_j
\le \sum_{j=1}^k\mu_j
\]
for every $k\ge 1$,
where missing parts are zero.
The \emph{Young diagram} of~$\lambda$ consists of left-justified rows
of lengths~$\lambda_1,\lambda_2,\dots$,
read from bottom to top in French notation.
A \emph{skew Young diagram}~$\lambda/\mu$ is the difference
of two nested Young diagrams.
A \emph{ribbon} is a connected skew Young diagram with no~$2\times2$ block of cells.

For background on symmetric functions,
see \citet{Mac95B} and \citet[Ch.~7]{Sta23B}.
Let~$\mathrm{Sym}$ be the algebra of symmetric functions over~$\mathbb Q$.
We write~$e_\lambda$ and~$s_\lambda$
for the elementary and Schur symmetric functions indexed by~$\lambda$,
respectively,
and~$h_r$ for the complete homogeneous symmetric function of degree~$r$.
For a composition~$I=i_1\dotsm i_\ell$,
put~$e_0=1$ and
\[
e_I
=e_{i_1}\dotsm e_{i_\ell},
\qquad e_\emptyset
=1.
\]
A symmetric function~$F$ is \emph{$e$-positive} or \emph{Schur positive}
if its coefficients in the corresponding basis are nonnegative,
written~$F\epos$ or~$F\spos$,
respectively.
We write~$F\gs H$ when~$F-H$ is Schur positive.
We also write~$G\spos$ for a graph~$G$ when~$X_G\spos$,
and use the analogous convention for~$e$-positivity.

The \emph{Littlewood--Richardson rule}
implies that the product of two Schur-positive symmetric functions
is Schur positive,
see \citet[Ch.~7]{Sta23B}.
A particular case, called \emph{Pieri's rule},
is as follows.
A \emph{vertical $k$-strip} is a skew Young diagram
with~$k$ cells and at most one cell in each row.

\begin{proposition}[Pieri rule]\label[pieri]{prop:Pieri}
For any integer~$k\ge 0$ and partition~$\lambda$,
\[
e_k s_\lambda
=\sum_{\text{$\mu\supseteq\lambda$, and $\mu/\lambda$ is a vertical $k$-strip}}
s_\mu.
\]
\end{proposition}

We will also use the following consequence of \cref{prop:Pieri}.

\begin{lemma}\label{lem:forced-cell-pieri}
For any partition~$\lambda$ and integer~$i>\ell(\lambda)$,
we have $e_{i-1}s_{\lambda1}\gs e_i s_\lambda$.
\end{lemma}

\begin{proof}
In view of \cref{prop:Pieri},
if~$\mu/\lambda$ is a vertical~$i$-strip,
then it uses a new row since~$i>\ell(\lambda)$.
Thus~$\mu$ contains the first cell in row~$\ell(\lambda)+1$,
so~$\mu/(\lambda,1)$ is a vertical~$(i-1)$-strip.
Both Pieri expansions are multiplicity-free by \cref{prop:Pieri},
so every Schur constituent on the right occurs on the left.
\end{proof}

For the disjoint union~$G\sqcup H$ of two graphs,
we have $X_{G\sqcup H}=X_G X_H$.
A \emph{stable composition} of a graph~$G$ of type~$\mu$
is a sequence~$(B_1,\dots,B_\ell)$
of pairwise disjoint nonempty stable sets whose union is~$V(G)$
and such that~$\abs{B_i}=\mu_i$ for every~$i$.
Let $\widetilde a_\mu(G)$ be the number of stable compositions of $G$ of type $\mu$.
\citet[Prop.~1.5]{Sta98} established the \emph{stable partition criterion}
saying that any Schur-positive graph
that has a stable partition of type $\lambda$
must have a stable partition of any type $\mu$ that is dominated by~$\lambda$.
One of its direct applications
pointed out by \citet[Thm.~2.3 and Eg.~2.4]{WW20}
will help prove \cref{thm:refined-bounds}.
A bipartite graph is \emph{balanced}
if it admits a bipartition whose two classes have sizes
differing by at most one.

\begin{theorem}[\citeauthor{WW20}]\label{thm:s+.bipartite}
Any Schur-positive bipartite graph is balanced.
In particular,
any spider with three odd legs is not Schur positive.
\end{theorem}

A \emph{ribbon tabloid} of shape~$\lambda$ is a tiling of its Young diagram
by ribbons that can be removed in some order
so that the cells remaining after each removal
form an ordinary Young diagram;
the removal order is not part of the tabloid.
It is said to be \emph{special}
if every ribbon meets the first column of the original diagram.
The \emph{height} of a ribbon tabloid~$T$ is the number
\[
h(T)
=\sum_{R\in T}(r(R)-1),
\]
where~$r(R)$ is the number of rows occupied by the ribbon~$R$.
\citet[Thm.~3.1]{WW20} gave a combinatorial formula
for each Schur coefficient of~$X_G$.

\begin{theorem}[\citeauthor{WW20}]\label{thm:s-coeff.XG}
Let~$G$ be a graph on~$n$ vertices.
For any partition~$\lambda\vdash n$,
\[
\Coeff{s_\lambda}{X_G}
=\sum_{\text{$T$: special ribbon tabloids of shape~$\lambda$}}(-1)^{h(T)}
\widetilde a_{\kappa_T}(G),
\]
where~$\kappa_T$ is the partition consisting of the ribbon lengths.
\end{theorem}

Another general formula for chromatic symmetric functions
is the arithmetic-progression identity
due to \citet[Proposition~6]{AWv24}.

\begin{proposition}[\citeauthor{AWv24}]
\label{prop:arithmetic-progression}
Let~$q\ge 1$,
let~$v_1,\dots,v_q$ be distinct vertices of a graph~$G$
with a common closed neighborhood,
and let~$w$ be another vertex,
adjacent to none of them.
Set~$G_0=G$,
and, for $1\le j\le q$,
let~$G_j$ be obtained from~$G_0$ by adding the edges~$wv_1,\dots,wv_j$.
Then
\[
X_{G_j}
=\frac{q-j}{q}X_{G_0}+\frac{j}{q}X_{G_q}
\qquad(0\le j\le q).
\]
In particular,
the chromatic symmetric functions~$X_{G_0},\dots,X_{G_q}$
form an arithmetic progression.
\end{proposition}

Let~$P_n$ be the path on~$n$ vertices.
Put~$X_n=X_{P_n}$ for~$n\ge 1$,
and set~$X_0=1$.
\citet[Prop.~5.3]{Sta95} obtained the following path generating function.

\begin{proposition}[\citeauthor{Sta95}]\label{prop:gf.path}
We have $\sum_{n\ge 0}X_n z^n=E(z)/D(z)$,
where
\[
E(z)
=\sum_{j\ge 0}e_j z^j
\quad\text{and}\quad
D(z)
=1-\sum_{j\ge 2}(j-1) e_j z^j.
\]
\end{proposition}

We use the following case of \citet[Lem.~4.4]{Zhe22}.

\begin{proposition}[\citeauthor{Zhe22}]\label{prop:3spider}
We have $X_{S(a,b,1)}
=X_{a+b+2}+e_1 X_{a+b+1}-X_{a+1} X_{b+1}$
for any~$a,b\in \mathbb Z_{\ge 1}$.
\end{proposition}

\citet[Cor.~3.2 and Thm.~4.2]{TW23X}
also established the following positivity results.

\begin{theorem}[\citeauthor{TW23X}]\label{thm:TW+}
We have $X_{a+2}-e_2 X_a\epos$ and~$S(a,2,1)\spos$ for any~$a\in \mathbb Z_{\ge 1}$.
\end{theorem}

While the latter assertion is recoverable from \cref{prop:s+.S1pqrs1},
the former follows from \cref{prop:gf.path}:
Multiplying the generating function by~$D(z)$ and comparing coefficients of~$z^n$,
we obtain
\[
X_n
=e_n+\sum_{i=1}^n(i-1) e_i X_{n-i}.
\]
For integers $1\le k\le n$,
it follows that
\[
X_n-(k-1) e_k X_{n-k}
=e_n+\sum_{1\le i\le n,\ i\ne k}(i-1) e_i X_{n-i}
\epos.
\]
Setting $k=2$ gives the former assertion in \cref{thm:TW+}.

A \emph{rooted graph} is a graph with a distinguished vertex called its root.
Let~$K_m$ be the complete graph on~$m$ vertices.
For a rooted graph~$G$ and an integer~$j\ge 0$,
let~$G^j$ be the graph obtained by attaching a path of length~$j$ at the root,
with all other vertices of the path new.
In particular, $G^0=G$.
Let~$P^0(G,K_m)$ be obtained
by identifying the root of~$G$ with one vertex of~$K_m$.
We use the following case of \citet[Prop.~3.1]{QTW26}.

\begin{proposition}[\citeauthor{QTW26}]\label{prop:KPG}
For any rooted graph~$G$ and integer~$m\ge 1$,
\[
X_{P^0(G,K_m)}
=(m-1)!\sum_{i=0}^{m-1}(1-i) e_i X_{G^{m-1-i}}.
\]
\end{proposition}

For integers~$g,h,k\ge 0$ and~$r,s,t\ge 1$,
let~$S^{ghk}_{rst}$ be the \emph{clique-spider}
obtained by joining a common center
to chosen roots of three disjoint cliques~$K_r,K_s,K_t$
by otherwise disjoint paths of lengths~$g,h,k$,
respectively.
A path of length~$0$ identifies its clique root with the center.
Each connecting path is a \emph{leg},
the corresponding clique is its \emph{foot},
and their union is a \emph{limb}.
We use the following clique-spider results
of \citet[Thm.~7.3, Cor.~7.4, and Thm.~8.3]{WW26X}.

\begin{theorem}[\citeauthor{WW26X}]\label{thm:s+:clique-spiders}
Let $r\ge s\ge t\ge 1$.
Then~$S^{100}_{rst}\spos$ and~$S^{010}_{rst}\spos$.
Moreover,
$S^{001}_{(t+1)(t+1)t}\spos$ for every integer~$t\ge 2$.
\end{theorem}

The last assertion follows from \citet[Thm.~8.3]{WW26X}:
for~$t=2$,
the triple~$(3,3,2)$ satisfies the criterion $(r,s)=(2t-1,\,t+1)$;
for~$t\ge 3$,
the triple $(t+1,\,t+1,\,t)$ satisfies the criterion $3t>r+s$
because $3t>2(t+1)$.

\section{Clique-spider positivity and propagation}\label{sec:master-recurrence}

We first establish the recurrences and clique-spider positivity results
used in the inductive proofs of \cref{thm:classification,thm:refined-bounds}.
The main identity is \cref{prop:rec.master},
which we call the \emph{master recurrence}.
For integers $b,c\ge 0$ and $a\ge 1$,
put
\[
\mathcal C_a(b,c)
=S^{1bc}_{a11}
\quad\text{and}\quad
C_a(b,c)
=X_{\mathcal C_a(b,c)}/(a-1)!.
\]
Then $C_a(p,q)=C_a(q,p)$.

\begin{proposition}[Master recurrence]\label{prop:rec.master}
For any spider $S(a,b,c)$,
\[
X_{S(a,b,c)}
=C_a(b,c)
+\sum_{i=2}^{a-1}(i-1) e_i X_{S(a-i,\,b,\,c)}.
\]
\end{proposition}
\begin{proof}
Let $H=S(1,b,c)$,
rooted at the leaf of the leg of length~$1$.
Then
\[
P^0(H,K_a)
=S^{1bc}_{a11}
\quad\text{and}\quad
H^{a-1-i}
=S(a-i,\,b,\,c)
\text{ for $0\le i\le a-1$}.
\]
Taking~$m=a$ in \cref{prop:KPG} gives
\[
X_{S^{1bc}_{a11}}
=(a-1)!\sum_{i=0}^{a-1}(1-i) e_i X_{S(a-i,\,b,\,c)}
=(a-1)!\brk3{
X_{S(a,b,c)}
-\sum_{i=2}^{a-1}(i-1) e_i X_{S(a-i,\,b,\,c)}}.
\]
Dividing by~$(a-1)!$ and rearranging gives the desired formula.
\end{proof}

We use \cref{prop:rec.master} together with the Schur positivity
of $C_a(b,c)$ and suitable lower bounds
for the Schur coefficients of the smaller spiders.
We will show the Schur positivity of a series of graphs,
starting from that of clique-spiders with leg length multiset $\{0, 0, 1\}$
given by \cref{thm:s+:clique-spiders}.

We will need the following relation among the clique-spiders~%
$S^{(g+1)hk}_{rst}$,
$S^{ghk}_{(r+1)st}$,
and $S^{ghk}_{1st}$.

\begin{lemma}[Clique-spider recurrence]\label{lem:rec:clique-spider}
For any $g,h,k\in \mathbb Z_{\ge 0}$ and $r,s,t\in \mathbb Z_{\ge 1}$,
\[
rX_{S^{(g+1)hk}_{rst}}
=X_{S^{ghk}_{(r+1)st}}
+(r-1) X_{K_r}X_{S^{ghk}_{1st}}.
\]
\end{lemma}

\begin{proof}
We shall apply \cref{prop:arithmetic-progression}.
In its notation, we let $G_0=K_r\sqcup S^{ghk}_{1st}$,
let~$v_1,\dots,v_r$ be the vertices of~$K_r$,
and let~$w$ be the vertex of the~$K_1$-foot.
Then $G_1\cong S^{(g+1)hk}_{rst}$ and $G_r\cong S^{ghk}_{(r+1)st}$.
Taking $(j,q)=(1,r)$ and using $X_{G\sqcup H}=X_G X_H$
gives the desired formula.
\end{proof}

We call a clique-spider \emph{special} if it is of the form $S_{1uv}^{ghk}$,
namely, if it has a vertex foot.
\cref{lem:relative-bootstrap} reduces the Schur positivity
of general clique-spiders to that of special ones.

\begin{lemma}[Reduction to special clique-spiders]\label{lem:relative-bootstrap}
Fix $g,h,k,L\in \mathbb Z_{\ge 0}$ and $s,u,v\in \mathbb Z_{\ge 1}$.
Set
\[
F_j
=S^{(g+j)hk}_{(L+s-j)uv}
\quad\text{and}\quad
D_j
=S^{(g+j)hk}_{1uv}.
\]
If the graphs $F_0,D_0,D_1,\dots,D_{L-1}$ are Schur positive,
then so are the graphs $F_1,\dots,F_L$.
\end{lemma}

\begin{proof}
Let $0\le j\le L-1$.
Applying \cref{lem:rec:clique-spider}
with first-leg length~$g+j$ and $r=L+s-j-1$ gives
\[
(L+s-j-1) X_{F_{j+1}}
=X_{F_j}+(L+s-j-2) X_{K_{L+s-j-1}}X_{D_j}.
\]
Note that $L+s-j\ge 2$.
Since $X_{K_n}\epos$,
induction on~$j$ proves the assertion.
\end{proof}

We next use this reduction to establish the Schur positivity
of the family~$S^{1pq}_{rs1}$ with $q<r$.

\begin{proposition}\label{prop:s+.S1pqrs1}
We have $S^{1pq}_{rs1}\spos$ for any integers $p\ge 0$,
$0\le q\le r-1$,
$r\ge 2$,
and $s\ge 1$.
\end{proposition}

\begin{proof}
We use strong induction on $n=p+q$,
simultaneously for all admissible~$r$ and $s$.
For~$n=0$,
the result follows from \cref{thm:s+:clique-spiders}.
Let~$n\ge 1$.
Fix admissible~$(p,q,r,s)$ with $p+q=n$.

First suppose~$p=0$.
We apply \cref{lem:relative-bootstrap} with parameters
\[
(g,h,k,L;s,u,v)
=(0,1,0,q;1,r,s).
\]
The graphs in that lemma are
\[
F_j
=S^{j10}_{(q+1-j)rs}\cong S^{10j}_{rs(q+1-j)}
\quad\text{and}\quad
D_j
=S^{j10}_{1rs}\cong S^{10j}_{rs1}.
\]
Since $q+1\le r$,
\cref{thm:s+:clique-spiders} gives Schur positivity
of the initial graph $F_0\cong S^{100}_{rs(q+1)}$.
For $0\le j<q$,
we have~$D_j\spos$ by the induction hypothesis.
Thus \cref{lem:relative-bootstrap} gives Schur positivity
of~$F_q\cong S^{10q}_{rs1}$.

Now suppose~$p\ge 1$.
We apply \cref{lem:relative-bootstrap} with
\[
(g,h,k,L;s,u,v)
=(0,1,q,p;s,r,1).
\]
The graphs in that lemma are
\[
F_j
=S^{j1q}_{(p+s-j)r1}\cong S^{1jq}_{r(p+s-j)1}
\quad\text{and}\quad
D_j
=S^{j1q}_{1r1}\cong S^{1jq}_{r11}.
\]
By the induction hypothesis,
the initial graph $F_0\cong S^{10q}_{r(p+s)1}$
and the graphs~$D_j$ with $0\le j<p$ are all Schur positive.
Hence $S^{1pq}_{rs1}\cong F_p\spos$.
\end{proof}

Note that for $(p,q,r,s)=(a,1,2,1)$,
\cref{prop:s+.S1pqrs1} recovers $X_{S(a,2,1)}\spos$ in \cref{thm:TW+}.
Next,
we establish the Schur positivity for the family $S^{1pq}_{rs1}$
without the restriction $q<r$ but with $s\le r+1$.

\begin{proposition}\label{prop:universal-core}
We have $S^{1pq}_{rs1}\spos$ 
for any integers~$p,q\ge 0$, 
$r\ge 2$,
and $1\le s\le r+1$.
In particular,
$C_a(p,q)\spos$ and~$S^{1pq}_{a31}\spos$ for any integers~$p,q\ge 0$ and~$a\ge 2$.
\end{proposition}

\begin{proof}
We use induction on $n=p+q$,
simultaneously for all admissible~$r,s$.
The cases $q<r$, including $n=0$,
follow from \cref{prop:s+.S1pqrs1}.
For the remaining cases,
assume $q\ge r$ and the result for every smaller value of $p+q$.
Order the pairs~$(p,s)$ and~$(q,1)$ as~$(x,\beta)$ and~$(y,\gamma)$
so that $x+\beta\ge y+\gamma$,
choosing $(x,\beta)=(q,1)$ in a tie.
Put
\[
R
=x+\beta-1
=\max(p+s-1,\,q).
\]
Then $R\ge r\ge 2$,
$y+\gamma\le R+1$,
and $x+y=p+q=n$.
Moreover, we have $x\ge 1$,
since otherwise~$x=0$ would imply $(x,\beta)=(p,s)$,
and the tie convention would give~$s>q+1$,
contradicting $s\le r+1\le q+1$.

We prove the Schur positivity of the following four graphs in order:
\[
G_0
=S^{010}_{(r+1)R(y+\gamma)},
\qquad G_1
=S^{110}_{rR(y+\gamma)},
\qquad G_2
=S^{11y}_{rR\gamma},
\qquad G_3
=S^{1xy}_{r\beta\gamma}\cong S^{1pq}_{rs1}.
\]
\begin{enumerate}
\item
If the marked foot~$K_R$ is not strictly smallest,
\cref{thm:s+:clique-spiders} applies after permuting the limbs.
Otherwise, $R<r+1$ and~$R<y+\gamma$.
Together with the preceding bounds,
these inequalities force~$R=r$ and~$y+\gamma=R+1$.
Thus $G_0\cong S^{001}_{(R+1)(R+1)R}\spos$ by \cref{thm:s+:clique-spiders}.

\item
We apply \cref{lem:relative-bootstrap} with
\[
(g,h,k,L;s,u,v)
=(0,1,0,1;r,\,R,\,y+\gamma).
\]
Then~$F_0=G_0\spos$,
and $D_0=S^{010}_{1R(y+\gamma)}\spos$ by \cref{thm:s+:clique-spiders}.
Hence~$G_1=F_1\spos$.

\item
If~$y=0$, then~$G_2=G_1\spos$.
For~$y\ge 1$,
we apply \cref{lem:relative-bootstrap} with
\[
(g,h,k,L;s,u,v)
=(0,1,1,y;\gamma,r,R).
\]
Then $F_0=S^{011}_{(y+\gamma)rR}\cong G_1\spos$,
and $D_j=S^{j11}_{1rR}\cong S^{11j}_{Rr1}$ for $0\le j<y$.
Since $j\le y-1\le R-\gamma\le R-1$, \cref{prop:s+.S1pqrs1} gives~$D_j\spos$.
It follows that~$G_2\cong F_y\spos$.

\item
If~$x=1$,
then~$R=\beta$ and~$G_3=G_2\spos$.
For~$x\ge 2$,
we apply \cref{lem:relative-bootstrap} with
\[
(g,h,k,L;s,u,v)
=(1,\,1,\,y,\,x-1;\,\beta,r,\gamma).
\]
Since $x-1+\beta=R$,
its initial graph is $F_0=S^{11y}_{Rr\gamma}\cong G_2\spos$.
Its auxiliary graphs satisfy
\[
D_{j-1}\cong S^{1jy}_{r1\gamma}
\qquad(1\le j<x).
\]
If~$\gamma=1$,
this belongs to the stated family with middle foot order~$1$.
Otherwise~$\gamma=s$ and $D_{j-1}\cong S^{1yj}_{rs1}$.
In either case, $j+y<x+y=n$,
and~$D_{j-1}\spos$ by induction.
Hence~$G_3\cong F_{x-1}\spos$.
\end{enumerate}
This completes the induction.
Taking~$s=1$ and~$s=3$ gives the two stated consequences.
\end{proof}

\begin{lemma}\label{lem:core-recurrence}
For any integers $p,q\ge 0$ and $a\ge 2$,
\[
C_a(p+2,\,q)
=e_2 C_a(p,q)
+\frac{1}{2(a-1)!}X_{S^{1pq}_{a31}}.
\]
As a consequence,
\[
C_a(p,q)
\gs \begin{dcases*}
e_2^{(p+q)/2-1}C_a(1,1), & if $p$ and $q$ are odd,
\\
e_2^{(p+q)/2}C_a(0,0), & if $p$ and $q$ are even,
\\
e_2^{(p+q-1)/2}C_a(1,0), & if $p+q$ is odd.
\end{dcases*}
\]
\end{lemma}

\begin{proof}
Applying \cref{lem:rec:clique-spider}
with parameters $(g,h,k;r,s,t)=(p,1,q;2,a,1)$,
we obtain
\[
2X_{S^{(p+1)1q}_{2a1}}
=X_{S^{p1q}_{3a1}}+X_{K_2}X_{S^{p1q}_{1a1}}.
\]
Since $S^{(p+1)1q}_{2a1}\cong S^{(p+2)1q}_{1a1}
\cong S^{1(p+2)q}_{a11}
=\mathcal C_a(p+2,\,q)$,
$S^{p1q}_{1a1}\cong \mathcal C_a(p,q)$, and $X_{K_2}=2e_2$,
we obtain
\[
2X_{\mathcal C_a(p+2,\,q)}
=X_{S^{1pq}_{a31}}+2e_2 X_{\mathcal C_a(p,q)}.
\]
Dividing by $2(a-1)!$ gives the desired formula.

By \cref{prop:universal-core}, $C_a(p+2,\,q)\gs e_2 C_a(p,q)$.
Note that $C_a(p,q)=C_a(q,p)$.
Whenever $p\ge 2$ or $q\ge 2$,
we can apply the inequality once.
Repeatedly using it,
we obtain the desired consequence.
\end{proof}

\cref{cor:propagation} reduces \cref{thm:classification}(iii)
to spiders with two even legs and one leg of length~$1$.

\begin{corollary}\label{cor:propagation}
Fix $b,c\ge 1$.
If $S(1,b,c)\spos$,
then $S(a,b,c)\spos$ for any $a\ge 1$.
\end{corollary}

\begin{proof}
We use induction on $a$,
with the case $a=1$ given by hypothesis.
Let $a\ge 2$.
The smaller spiders in \cref{prop:rec.master} are Schur positive
by the induction hypothesis,
and $C_a(b,c)\spos$ by \cref{prop:universal-core}.
Thus the recurrence gives $X_{S(a,b,c)}\spos$.
\end{proof}

\section{Critical coefficients and path-difference bounds}
\label{sec:coefficient-bounds}

This section provides the coefficient formulas and path estimates
used in \cref{sec:proof-two-odd,sec:proof-two-even}.
We first determine two Schur coefficients of spiders with exactly one even leg.
One of these coefficients gives the necessary condition
in \cref{thm:classification}(ii).
We then prove Schur bounds for path differences,
which enter the initial cases of the later inductions.

\subsection{Critical coefficients of spiders}\label{sec:s-}

Throughout this paper, for any spider graph,
we use~$N$ to denote the total number of its vertices,
and put $h=\floor{N/2}$.
Define $\delta_a(x,y)=\ceil{(5x+5y+2-a)/2}$.

\begin{proposition}\label{prop:critical}
Let~$a\ge 2$ be even and~$b,c\ge 1$ be odd,
and put $h=(a+b+c)/2$.
Then
\[
\Coeff{s_{(h+1)h}}{X_{S(a,b,c)}}
=1
\quad\text{and}\quad
\Coeff{s_{hh1}}{X_{S(a,b,c)}}
=\delta_a(b,c).
\]
As a consequence, if $S(a,b,c)\spos$,
then $a\le 5b+5c+2$.
\end{proposition}

\begin{proof}
Let $G=S(a,b,c)$.
Suppose that $a=2x$,
$b=2y+1$,
and $c=2z+1$,
where $x\in \mathbb Z_{\ge 1}$ and $y,z\in \mathbb Z_{\ge 0}$.
Any stable set containing the center of $G$ has size at most $1+x+y+z=h$,
whereas a stable set avoiding the center has size at most
\[
x+(y+1)+(z+1)
=h+1.
\]
Hence the independence number of~$G$ is $\alpha(G)=h+1$.

Let~$A\sqcup B$ be the bipartition of $G$ chosen
so that the center belongs to~$A$.
Then $\abs{A}=h$ and $\abs{B}=h+1$.
By \cref{thm:s-coeff.XG}, $\Coeff{s_{(h+1)h}}{X_G}
=\widetilde a_{(h+1)h}=1$,
and
\begin{equation}\label{pf:s.rr1.Sebc}
\Coeff{s_{hh1}}{X_G}
=\widetilde a_{hh1}
-\widetilde a_{(h+1)(h-1)1}
-\widetilde a_{(h+1)h}.
\end{equation}

It remains to calculate $\widetilde a_{hh1}$ and $\widetilde a_{(h+1)(h-1)1}$.
For any vertex~$v$ and integer~$i$,
let~$N_i(v)$ count the stable compositions~$(B_1,B_2)$ of~$G-v$
with $\abs{B_1}-\abs{B_2}=i$.
Then
\[
\widetilde a_{hh1}
=\sum_{v\in V(G)}N_0(v)
\quad\text{and}\quad
\widetilde a_{(h+1)(h-1)1}
=\sum_{v\in V(G)}N_2(v).
\]
Let~$H_1,\dots,H_k$ be the components of~$G-v$.
Then each $H_i$ is bipartite.
Let~$d_j$ be the absolute difference
between the sizes of the two bipartition classes of~$H_j$.
Then
\[
N_i(v)
=\Coeff{t^i}{P_v(t)},
\quad\text{where}\quad
P_v(t)
=\prod_{j=1}^k(t^{d_j}+t^{-d_j}).
\]
Classifying the vertices of $G$ by position
gives \cref{tab:critical-count-table}.
All distances are measured from the center.
The last three rows apply separately with $u=y$ and $u=z$.

\begin{table}[htbp]
\centering
\caption{Stable-composition counts $N_i(v)=\Coeff{t^i}{P_v(t)}$ by the position of~$v$.}
\label{tab:critical-count-table}
\renewcommand{\arraystretch}{1.1}
\begin{tabular}{@{}c@{\quad}>{$}c<{$}@{\quad}>{$}c<{$}@{\quad}>{$}c<{$}@{\quad}>{$}c<{$}@{}}
\toprule
Position of~$v$&\multicolumn{1}{c}{Number of choices of~$v$}&P_v(t)&N_2(v)&N_0(v)
\\
\midrule
Center&1&2(t+t^{-1})^2&2&4
\\
$2x$-leg,
odd distance&x&(t+t^{-1})^2&1&2
\\
$2x$-leg,
even distance,
non-leaf&x-1&2(t^2+t^{-2})&2&0
\\
Leaf of the~$2x$-leg&1&t^2+t^{-2}&1&0
\\
$(2u+1)$-leg,
even distance&u&(t+t^{-1})^2&1&2
\\
$(2u+1)$-leg,
odd distance,
non-leaf&u&4&0&4
\\
Leaf of a~$(2u+1)$-leg&1&2&0&2
\\
\bottomrule
\end{tabular}
\end{table}

Each factor~$2$ appearing in the column labeled by~$P_v(t)$
arises from the corresponding fact~$d_j=0$,
which yields the Laurent polynomial $t^0+t^{-0}=2$.
Summing the contributions in the last two columns gives
\begin{align*}
\widetilde a_{hh1}
&=4+2x+2(y+z)+4(y+z)+2\cdot 2
=2(x+3y+3z+4),
\quad\text{and}
\\
\widetilde a_{(h+1)(h-1)1}
&=2+x+2(x-1)+1+(y+z)
=3x+y+z+1.
\end{align*}
Substituting them into \cref{pf:s.rr1.Sebc},
we obtain
\[
\Coeff{s_{hh1}}{X_G}
=2(x+3y+3z+4)-(3x+y+z+1)-1
=(5b+5c+2-a)/2.
\]
The consequence follows since $\Coeff{s_{hh1}}{X_G}<0$ when $a>5b+5c+2$.
\end{proof}

\subsection{Schur bounds for path differences}\label{sec:path-bridge}

Define the symmetric function $\Delta_n$ by $\Delta_0=0$,
$\Delta_1=e_1$, and for $n\ge 2$,
\[
\Delta_n
=e_1 X_{n-1}-e_2 X_{n-2}.
\]
Then $\Delta_2=s_2$ and $\Delta_3=e_{21}$.
The connection with spiders follows from \cref{prop:3spider}:
\[
X_{S(n-3,\,1,\,1)}
=\bigl(X_n-e_2 X_{n-2}\bigr)+\Delta_n
\qquad(n\ge 4).
\]
The first summand is~$e$-positive by \cref{thm:TW+}.
Thus Schur bounds for~$\Delta_n$ provide estimates
for the initial family~$S(a,1,1)$ used in the later inductions.
The goal of this subsection is to establish \cref{thm:path-bridge}.

\begin{theorem}
\label{thm:path-bridge}
For every integer $h\ge 2$,
we have
\[
\Coeff{s_{hh}}{\Delta_{2h}}
=-1
\quad\text{and}\quad
\Coeff{s_{hh1}}{\Delta_{2h+1}}
=3-h.
\]
Every other Schur coefficient of $\Delta_{2h}$ and $\Delta_{2h+1}$ is nonnegative.
Moreover,
\[
\Coeff{s_{(h+1)(h-1)}}{\Delta_{2h}}
=1
\quad\text{and}\quad
\Coeff{s_{(h+1)h}}{\Delta_{2h+1}}
=1.
\]
\end{theorem}

Define the \emph{boundary part} of~$\Delta_N$ by
\[
Q_N
=\sum_{\lambda\vdash N, \
\operatorname{tail}_2(\lambda)\le 1}
\brk1{\Coeff{s_\lambda}{\Delta_N}}s_\lambda,
\quad\text{where $\operatorname{tail}_2(\lambda)
=\sum_{j\ge 3}\lambda_j$.}
\]
The proof has three stages.
First,
\cref{lem:path-boundary} determines all Schur coefficients indexed
by partitions with tail at most one.
These include the exceptional coefficients in \cref{thm:path-bridge}.
Next,
we subtract this boundary part and use \cref{prop:gf.path}
to derive a recurrence for the remainder.
\Cref{lem:path-forcing-positive}
proves that the new terms in this recurrence are Schur positive.
Finally,
the recurrence expresses each remainder as a sum of products
of these Schur-positive terms with~$e$-positive symmetric functions,
which proves that all remaining coefficients are nonnegative.

For~$F\in \mathrm{Sym}$ and $d\in \mathbb Z_{\ge 1}$,
write $F^{[d]}=F(x_1,\dots,x_d,0,0,\dots)$.

\begin{lemma}\label{lem:3var}
We have the following identities for symmetric functions in $3$ variables.
\begin{enumerate}
\item
We have $s_{abc}^{[3]}
=(e_3^{[3]})^c s_{(a-c)(b-c)}^{[3]}$
for any integers $a\ge b\ge c\ge 0$.
\item
We have $e_1^{[3]} s_{h(h-1)}^{[3]}
=s_{(h+1)(h-1)}^{[3]}
+e_2^{[3]} s_{(h-1)^2}^{[3]}$
for any integer $h\ge 1$.
\item
The generating function of $s_{hh}^{[3]}$ is
\[
\sum_{h\ge 0}s_{hh}^{[3]}t^h
=\frac1{1-e_2^{[3]}t+e_{31}^{[3]}t^2-e_{33}^{[3]}t^3}.
\]
\end{enumerate}
\end{lemma}

\begin{proof}
The first desired formula follows directly from \cref{prop:Pieri},
since the number of rows in any resulting shape
is at most the number of $3$ of variables.
For $h=1$,
the second identity is
\[
e_1^{[3]}s_1^{[3]}
=s_2^{[3]}+e_2^{[3]}.
\]
For $h\ge 2$,
adding one cell to $(h,\,h-1)$ gives~$(h+1,\,h-1)$,
$h^2$ or $(h,\,h-1,\,1)$,
the latter two of which can be obtained alternatively
by adding a vertical~$2$-strip to~$(h-1)^2$;
this proves the second formula.

By the combinatorial interpretation,
$s_{hh}^{[3]}$ is the weight sum
of semistandard Young tableaux of shape~$(h,h)$ with entries~$1$,
$2$,
or~$3$.
Each column of such a tableau,
in French notation from bottom to top,
contributes a factor~$y_1=x_1x_2$,
$y_2=x_1x_3$,
or~$y_3=x_2x_3$.
From left to right,
the columns are weakly increasing in the order $y_1<y_2<y_3$.
By the combinatorial interpretation
of the complete homogeneous generating function,
\[
\sum_{h\ge 0}s_{hh}^{[3]}t^h
=\prod_{i=1}^3\frac1{1-y_it}
=\frac1{1-c_1t+c_2t^2-c_3t^3},
\]
where $c_1=e_1(y_1,y_2,y_3)=e_2^{[3]}$,
$c_2=e_2(y_1,y_2,y_3)=e_{31}^{[3]}$,
and $c_3=e_3(y_1,y_2,y_3)=e_{33}^{[3]}$.
\end{proof}

Combining these identities with two-variable specialization,
we obtain an explicit formula for the boundary part~$Q_N$.

\begin{lemma}\label{lem:path-boundary}
For $N\le 3$, we have $Q_N=\Delta_N$.
For~$h\ge 2$,
\begin{align*}
Q_{2h}
&=s_{(h+1)(h-1)}-s_{hh}
+(h-2)s_{(h+1)(h-2)1}
+(2h-1)s_{h(h-1)1},
\quad\text{and}
\\
Q_{2h+1}
&=s_{(h+1)h}+3(h-1) s_{(h+1)(h-1)1}
-(h-3) s_{hh1},
\end{align*}
where the term $(h-2)s_{(h+1)(h-2)1}$ is omitted when $h=2$.
\end{lemma}

\begin{proof}
We determine the two-row terms first.
Since a proper coloring of a nonempty path using two colors alternates,
\[
X_{2k}^{[2]}
=2(x_1x_2)^k\quad(k\ge 1)
\quad\text{and}\quad
X_{2k+1}^{[2]}
=(x_1+x_2)(x_1x_2)^k\quad(k\ge 0).
\]
Substituting them into the definition $\Delta_n=e_1 X_{n-1}-e_2 X_{n-2}$,
we obtain
\begin{align*}
\Delta_{2h}^{[2]}
&=(x_1+x_2)^2(x_1x_2)^{h-1}-2(x_1x_2)^h
=(x_1x_2)^{h-1}(h_2^{[2]}-x_1x_2)
=s_{(h+1)(h-1)}^{[2]}-s_{hh}^{[2]},\quad\text{and}
\\
\Delta_{2h+1}^{[2]}
&=2(x_1+x_2)(x_1x_2)^h
-x_1x_2(x_1+x_2)(x_1x_2)^{h-1}
=(x_1+x_2)(x_1x_2)^h
=s_{(h+1)h}^{[2]}.
\end{align*}
Since two-variable Schur polynomials are linearly independent,
these identities determine all two-row coefficients.
It remains to determine the coefficients
indexed by a partition of the form~$(a,b,1)$.

By \cref{prop:gf.path},
\begin{equation}\label{eq:path-difference-series}
\sum_{n\ge 0}\Delta_nz^n
=\Delta_1z+\sum_{n\ge 2}(e_1 X_{n-1}-e_2 X_{n-2})z^n
=(e_1 z-e_2 z^2)\frac{E(z)}{D(z)}.
\end{equation}
Put $A=e_1^{[3]}$,
$B=e_2^{[3]}$,
and $C=e_3^{[3]}$.
Specializing to three variables gives
\begin{equation}\label{eq:path-three-variable-series}
\sum_{n\ge 0}\Delta_n^{[3]}z^n
=\frac{(Az-Bz^2)
(1+Az+Bz^2+Cz^3)}
{1-Bz^2-2Cz^3}.
\end{equation}
For $F,G\in \mathbb Q[A,B,C]$,
we write
\[
F\equiv G\pmod{C^2}
\quad\Longleftrightarrow\quad
F-G
=C^2J
\text{ for some }J\in \mathbb Q[A,B,C].
\]
We use the same notation for formal power series coefficient by coefficient.

By \cref{lem:3var}(1),
the ideal $(C^2)$ in $\mathbb Q[A,B,C]$
is spanned by the three-variable Schur polynomials
whose third part is at least $2$.
Thus reduction modulo $C^2$ preserves the coefficients indexed
by partitions with third part $0$ or $1$,
including those of the form $(a,b,1)$.
Substituting
\begin{align*}
\frac1{1-Bz^2-2Cz^3}
&=\frac1{1-Bz^2}
\cdot\frac1{1-2Cz^3/(1-Bz^2)}
\equiv \frac1{1-Bz^2}
+\frac{2Cz^3}{(1-Bz^2)^2}
\\
&=\sum_{k\ge 0}B^kz^{2k}
+2C\sum_{k\ge 0}(k+1)B^kz^{2k+3}
\pmod{C^2}
\end{align*}
into \cref{eq:path-three-variable-series} and extracting the coefficient of~$z^n$,
we obtain
\begin{align}
\label{eq:path-even-modular}
\Delta_{2h}^{[3]}
&\equiv \bigl(A^2-2B\bigr)B^{h-1}
+(2h-1)AB^{h-2}C
\pmod{C^2},
\quad\text{and}
\\
\Delta_{2h+1}^{[3]}
&\equiv AB^h
+\brk1{2(h-1)A^2-(4h-5)B}B^{h-2}C
\pmod{C^2}.
\label{eq:path-odd-modular}
\end{align}

We must subtract the full two-row Schur functions
before extracting the terms linear in~$C$.
The generating function in \cref{lem:3var}
gives $s_{00}^{[3]}=1$ and~$s_{11}^{[3]}=B$,
and for~$h\ge 2$,
\[
s_{hh}^{[3]}
=[t^h]\frac1{1-Bt+ACt^2-C^2t^3}
\equiv[t^h]\brk3{
\frac1{1-Bt}
-\frac{ACt^2}{(1-Bt)^2}}
=B^h-(h-1)AB^{h-2}C
\pmod{C^2}.
\]
By \cref{lem:3var},
we can express the two-row terms as
\begin{align}
\notag
s_{(h+1)h}^{[3]}
&=As_{hh}^{[3]}-Cs_{(h-1)^2}^{[3]}
\\
\label{pf:s.r+1.r.[3]}
&\equiv
AB^h-\brk1{(h-1)A^2+B}B^{h-2}C
\pmod{C^2},
\quad\text{and}
\\
\notag
s_{(h+1)(h-1)}^{[3]}
&=As_{h(h-1)}^{[3]}-Bs_{(h-1)^2}^{[3]}
\\
\label{pf:s.r+1.r-1.[3]}
&\equiv (A^2-B)B^{h-1}
+\brk1{(h-3)B-(h-2)A^2}AB^{h-3}C
\pmod{C^2}.
\end{align}
The congruences hold for~$h\ge 2$ and~$h\ge 3$,
respectively.

For~$h\ge 3$,
substituting \cref{pf:s.r+1.r-1.[3]} into \cref{eq:path-even-modular} gives
\begin{align*}
\left.\frac{\Delta_{2h}^{[3]}-s_{(h+1)(h-1)}^{[3]}+s_{hh}^{[3]}}
{C}\right|_{x_3=0}
&=(h-2)(x_1+x_2)^3(x_1x_2)^{h-3}
+3(x_1+x_2)(x_1x_2)^{h-2}
\\
&=(h-2) s_{h(h-3)}^{[2]}
+(2h-1) s_{(h-1)(h-2)}^{[2]}.
\end{align*}
After subtracting the complete two-row part,
every remaining three-variable Schur term is $s_{abc}^{[3]}$ with $c\ge 1$.
By \cref{lem:3var}, $s_{abc}^{[3]}=C^c s_{(a-c)(b-c)}^{[3]}$.
After division by $C$ and specialization at $x_3=0$,
the terms with $c\ge 2$ vanish,
whereas $s_{ab1}^{[3]}/C$ becomes $s_{(a-1)(b-1)}^{[2]}$.
Conversely,
we can recover the terms
\[
(h-2)s_{(h+1)(h-2)1}+(2h-1)s_{h(h-1)1}
\]
in~$\Delta_{2h}$ from the formula above.
This proves the desired formula for $Q_{2h}$ with $h\ge 3$.
It is direct to calculate $\Delta_4
=s_{31}-s_{22}+3s_{211}+2s_{1^4}$,
which proves the desired even formula at~$h=2$.

Similarly, for~$h\ge 2$,
substituting \cref{pf:s.r+1.r.[3]} into \cref{eq:path-odd-modular} gives
\begin{align*}
\left.\frac{\Delta_{2h+1}^{[3]}-s_{(h+1)h}^{[3]}}
{C}\right|_{x_3=0}
&=3(h-1)(x_1+x_2)^2(x_1x_2)^{h-2}
-(4h-6)(x_1x_2)^{h-1}
\\
&=3(h-1) s_{h(h-2)}^{[2]}
-(h-3) s_{(h-1)^2}^{[2]}.
\end{align*}
We can recover the terms
\[
3(h-1) s_{(h+1)(h-1)1}-(h-3) s_{hh1}
\]
in $\Delta_{2h+1}$ from the formula above.
This proves the desired formula for $Q_{2h+1}$.
\end{proof}

Put $H_n=\Delta_n-Q_n$.
Then $H_0=H_1=H_2=H_3=0$.
Multiplying \cref{eq:path-difference-series} by~$D(z)$ gives
\[
\left(1-\sum_{i\ge 2}(i-1) e_i z^i\right)
\sum_{m\ge 0}\Delta_mz^m
=(e_1 z-e_2 z^2)\sum_{j\ge 0}e_j z^j.
\]
For~$n\ge 4$,
comparing coefficients of~$z^n$ yields
\[
\Delta_n-\sum_{i=2}^{n}(i-1) e_i\Delta_{n-i}
=e_1 e_{n-1}-e_2 e_{n-2}.
\]
Substituting $\Delta_m=H_m+Q_m$ and isolating~$H_n$ therefore gives,
for $n\ge 4$,
\begin{equation}\label{eq:path-remainder-recurrence}
H_n
=G_n+\sum_{i=2}^{n-4}(i-1) e_i H_{n-i},
\end{equation}
where
\begin{equation}\creflabel[def]{def:Gn}
G_n
=e_1 e_{n-1}-e_2 e_{n-2}-Q_n
+\sum_{i=2}^{n-1}(i-1) e_i Q_{n-i}.
\end{equation}

\begin{lemma}\label{lem:path-forcing-positive}
For every~$n\ge 4$, write $G_n=W_n+V_n$,
where~$W_n$ is the part of the Schur expansion of~$G_n$
consisting of the terms indexed by partitions
with largest part at most~$2$.
Then $W_4=2e_4$, $W_5=7s_{21^3}+4e_5$,
and for $n\ge 6$,
\[
W_n
=2(n-3) e_n
+7(n-4) s_{21^{n-2}}
+5(n-5)\brk1{s_{2^21^{n-4}}+s_{2^31^{n-6}}}.
\]
Moreover, $V_n\spos$ and $G_n\spos$.
\end{lemma}

\begin{proof}
By \cref{lem:path-boundary},
it is direct to compute that $Q_1$,
$\dots$,
$Q_5$ are,
respectively,
\[
s_1,\quad
s_2,\quad
e_{21},\quad
s_{31}-s_{22}+3s_{211},\quad\text{and}\quad
s_{32}+3s_{311}+s_{221}.
\]
By definition,
it is direct to calculate that
\begin{align*}
G_4
&=e_{31}-e_{22}-Q_4+e_2 Q_2+2e_3 Q_1
=2s_{1^4},\quad\text{and}
\\
G_5
&=e_{41}-e_{32}-Q_5
+e_2 Q_3
+2e_3 Q_2
+3e_4 Q_1
=7s_{21^3}+4s_{1^5}.
\end{align*}
Thus $W_4=G_4$,
$W_5=G_5$,
and $V_4=V_5=0$.
For the remainder of the proof,
let~$n\ge 6$.

For a partition~$\lambda$,
define its \emph{width} to be its maximum part~$\lambda_1$.
For every~$m\ge 6$,
each Schur constituent of~$Q_m$ has width at least~$3$.
Since adding a vertical strip cannot decrease width,
only $Q_1,\dots,Q_5$ contribute to~$W_n$.
We apply \cref{prop:Pieri} to \cref{def:Gn}.
For completeness,
\cref{tab:path-width-two} lists every contribution to the coefficient
indexed by~$(2^j,\,1^{n-2j})$.
\begin{table}[htbp]
\centering
\caption{Contributions to the width-two parts of~$G_n$ for~$n\ge 6$.}
\label{tab:path-width-two}
\begin{tabular}{>{$}l<{$}*{4}{>{$}c<{$}}}
\toprule
\text{Term }s_{2^j 1^{n-2j}} & j=0 & j=1 & j=2 & j=3
\\
\midrule
e_1 e_{n-1}-e_2 e_{n-2} & & & -1 &
\\
(n-2) e_{n-1}Q_1 & n-2 & n-2 & &
\\
(n-3) e_{n-2}Q_2 & & n-3 & &
\\
(n-4) e_{n-3}Q_3 & n-4 & 2(n-4) & 2(n-4) & n-4
\\
(n-5) e_{n-4}Q_4 & & 3(n-5) & 2(n-5) & 3(n-5)
\\
(n-6) e_{n-5}Q_5 & & & n-6 & n-6
\\
\bottomrule
\end{tabular}
\end{table}
A blank entry means zero.
No term has~$j\ge 4$,
since every starting shape has at most three rows
and a vertical strip creates new rows only of length~$1$.
Summing its columns gives the desired formula for $W_n$.
Thus $W_n\spos$.

It remains to prove~$V_n\spos$.
The \emph{Schur support} of a symmetric function~$F$ is
\[
\operatorname{supp}_s(F)
=\{\lambda\colon\Coeff{s_\lambda}{F}\ne 0\}.
\]
By the definition of~$H_n$,
\[
\operatorname{supp}_s(H_n)
\subseteq \{\lambda\vdash n\colon\operatorname{tail}_2(\lambda)\ge 2\}.
\]
By \cref{prop:Pieri},
multiplying $e_i$ preserves the property~$\operatorname{tail}_2\ge 2$.
Thus \cref{eq:path-remainder-recurrence} implies
\[
\operatorname{supp}_s(G_n)
\subseteq \{\lambda\vdash n\colon\operatorname{tail}_2(\lambda)\ge 2\}.
\]
In view of \cref{def:Gn},
$e_1 e_{n-1}-e_2 e_{n-2}$ contributes nothing
since a contributing term would have width at least $3$,
and neither does $-Q_n$ since a contributing term would have tail at most $1$.
By \cref{lem:path-boundary},
the only negative summands left in \cref{def:Gn} are
\[
-(i-1) e_i s_{hh}\quad(h\ge 2)
\quad\text{and}\quad
-(i-1)(h-3) e_i s_{hh1}\quad(h\ge 4),
\]
with~$i\ge 2$ and total degree~$n$.
Consequently,
\cref{prop:Pieri} confines every possibly negative coefficient of~$V_n$
to the four families in \cref{tab:path-negative-families}.
\begin{table}[htbp]
\centering
\caption{Coefficients at every possible negative location of~$V_n$.}
\label{tab:path-negative-families}
\setlength{\tabcolsep}{12pt}
\begin{tabular}{*{3}{>{$}l<{$}}}
\toprule
\text{Partition }\lambda & \text{Range} & \Coeff{s_\lambda}{G_{|\lambda|}}
\\
\midrule
(a,a,1^b) & a\ge 3,\,b\ge 2 & 5(b-1)
\\
(a+1,\,a,\,1^b) & a\ge 2,\,b\ge 2 & (6a+2)(b-1)
\\
(a,a,2,1^b) & a\ge 3,\,b\ge 0 & 6b+7
\\
(a+1,\,a,\,2,\,1^b) & a\ge 2,\,b\ge 0 & 6a(b+1)-2
\\
\bottomrule
\end{tabular}
\end{table}
We now compute the entries in the last column of \cref{tab:path-negative-families}.
We call each Schur term~$c_\mu s_\mu$ in the expansion of~$Q_n$
in \cref{lem:path-boundary} a \emph{seed},
including its coefficient~$c_\mu$.
In each product $(i-1) e_i Q_{n-i}$,
\cref{prop:Pieri} shows that a seed~$c_\mu s_\mu$
contributes $(i-1)c_\mu$ to the coefficient of~$s_\lambda$
if~$\lambda/\mu$ is a vertical strip of size~$i$,
and contributes zero otherwise.
The seed $(h-2)s_{(h+1)(h-2)1}$
contributes to none of the four families in \cref{tab:path-negative-families}.
The contributions are listed in \cref{tab:path-seed-contributions}.
\begin{table}[htbp]
\centering
\caption{Contributions of the seeds to the four coefficients in
\cref{tab:path-negative-families}.}
\label{tab:path-seed-contributions}
\setlength{\tabcolsep}{7pt}
\renewcommand{\arraystretch}{1.3}
\begin{tabular}{@{}>{$}r<{$}@{}>{$}l<{$}*{4}{>{$}c<{$}}@{}}
\toprule
\multicolumn{2}{@{}c}{Seed} & (a,a,1^b) & (a+1,\,a,\,1^b)
& (a,a,2,1^b) & (a+1,\,a,\,2,\,1^b)
\\
\midrule
& s_{(a+1)(a-1)} & & b & &
\\
- & s_{(a-1)^2} & -(b+1) & & &
\\
- & s_{aa} & -(b-1) & -b & &
\\
(2a-1) & s_{a(a-1)1} & (2a-1)(b-1) & (2a-1)b
& (2a-1)(b+1) & (2a-1)(b+2)
\\
(2a+1) & s_{(a+1)a1} & & (2a+1)(b-2) & & (2a+1)b
\\
& s_{a(a-1)} & b & b+1 & &
\\
& s_{(a+1)a} & & b-1 & &
\\
3(a-1) & s_{(a+1)(a-1)1} & & 3(a-1)(b-1)
& & 3(a-1)(b+1)
\\
(4-a) & s_{(a-1)^21} & (4-a)b & & (4-a)(b+2) &
\\
(3-a) & s_{aa1} & (3-a)(b-2) & (3-a)(b-1)
& (3-a)b & (3-a)(b+1)
\\
\bottomrule
\end{tabular}
\end{table}

Summing the entries in each column of \cref{tab:path-seed-contributions}
gives the corresponding formula in \cref{tab:path-negative-families}.
These coefficients are positive in the stated ranges.
Every other coefficient of $V_n$ receives contributions only
from nonnegative seeds.
Hence $V_n\spos$.
\end{proof}

Now we are in a position to prove \cref{thm:path-bridge}.

\begin{proof}[Proof of \cref{thm:path-bridge}]
We prove $H_n\spos$ by induction on $n$.
For $0\le n\le 3$, we have $H_n=0$.
For $n\ge 4$,
\cref{eq:path-remainder-recurrence} expresses $H_n$
as a sum of $G_n$ and nonnegative multiples of $e_i H_{n-i}$
with smaller indices.
By \cref{lem:path-forcing-positive},
the induction hypothesis,
and \cref{prop:Pieri},
every summand is Schur positive.
Thus $H_n\spos$.

By definition,
every Schur constituent of $H_n$ is indexed
by a partition $\lambda$ with $\operatorname{tail}_2(\lambda)\ge 2$.
Hence the boundary coefficients of $\Delta_n=H_n+Q_n$
are those in \cref{lem:path-boundary}.
That lemma gives the asserted exceptional coefficients
and the two coefficients equal to $1$.
All other coefficients of $Q_n$ are nonnegative by \cref{lem:path-boundary}.
The coefficients indexed by partitions $\lambda$
with $\operatorname{tail}_2(\lambda)\ge 2$ are those of $H_n$
and are therefore also nonnegative.
\end{proof}

\section{Schur bounds for spiders with two odd legs}\label{sec:proof-two-odd}

Throughout this section,
let $b$ and $c$ be positive odd integers.
The goal of this section is to establish \cref{thm:threshold-engine},
in which the odd $a$ assertion implies \cref{thm:refined-bounds}(i),
while the even $a$ assertion,
together with \cref{prop:critical}, gives \cref{thm:classification}(ii).

\begin{theorem}\label{thm:threshold-engine}
Let~$b$ and $c$ be positive odd integers,
and $a\in \mathbb Z_{\ge 1}$.
Assume that $a$ is odd or $\delta_a(b,c)\ge 0$.
Then $f_a(b,c)\spos$,
where
\[
f_a(b,c)
=\begin{dcases*}
X_{S(a,b,c)}-s_{(h+1)(h-1)}+s_{hh}, & if $a$ is odd,
\\
X_{S(a,b,c)}-s_{(h+1)h}-\delta_a(b,c) s_{hh1}, & if $a$ is even and if $\delta_a(b,c)\ge 0$.
\end{dcases*}
\]
\end{theorem}

The proof begins with the initial families~$f_a(1,1)$ for odd~$a$ and~$f_2(b,c)$.
By symmetry,
the remaining odd-leg cases have~$a\ge 3$.
\Cref{lem:s+:C3} supplies the additional Schur-positive term
needed for the case $a=3$.
The estimates in \cref{lem:even-block,lem:odd-block}
then handle even~$a\ge 4$ and odd~$a\ge 5$,
respectively,
before the proof of \cref{thm:threshold-engine} combines these results.
For~$N\ge 4$,
put~$h=\floor{N/2}$ and define
\[
\mathcal H_N
=s_{(h-1)^2 1^{N-2h+2}},
\qquad R_h
=s_{hh},
\qquad B_h
=s_{(h+1)h},
\qquad T_h
=s_{hh1},
\qquad U_h
=s_{(h+1)(h-1)}.
\]
Omitting trailing zero parts gives~$R_0=1$,
$B_0=s_1$,
and~$U_1=s_2$.
We set~$T_0=0$.
We first isolate a Schur-positive term in~$C_3(b,c)$
that will compensate for the negative term in the odd-leg induction
when~$a=3$.

\begin{lemma}\label{lem:s+:C3}
Let $b$ and $c$ be positive odd integers,
and put $h=(b+c+4)/2$.
Then $C_3(b,c)\gs 11\mathcal H_{2h}$.
\end{lemma}

\begin{proof}
It is direct to compute that
\[
C_3(1,1)
=2s_{321}
+5s_{3111}
+8s_{222}
+11s_{2211}
+23s_{21^4}
+24s_{1^6}
\gs 11s_{2^21^2}.
\]
By \cref{lem:core-recurrence},
\[
C_3(b,c)
\gs e_2^{h-3}C_3(1,1)
\gs 11e_2^{h-3}s_{2^21^2}
\gs 11 \mathcal H_{2h}.
\]
The last step follows from \cref{prop:Pieri}
by repeatedly adding a vertical~$2$-strip to the first two rows.
\end{proof}

For even~$a\ge 4$,
the master recurrence leads to the following estimate
for the paired terms after applying the induction hypotheses.

\begin{lemma}\label{lem:even-block}
Let $b$ and $c$ be positive odd integers.
Let $a\ge 4$ be an even integer,
and $h=(a+b+c)/2$.
If $\delta_a(b,c)\ge 0$, then $M_E\spos$,
where $M_E=\sum_{j=0}^{a/2-1}L_j$,
with
\[
L_j
=(2j-1) e_{2j}
(B_{h-j}+\delta_{a-2j}(b,c)T_{h-j})
-2j\,e_{2j+1}R_{h-j}.
\]
\end{lemma}

\begin{proof}
Here~$N=2h+1$, so $\mathcal H_{2h+1}=s_{(h-1)^21^3}$.
Note that $a/2-1\ge 1$ and
\[
h-j
\ge h-(a/2-1)
=(b+c)/2+1
\ge 2.
\]
Taking~$\lambda=(h-j)^2$ and $i=2j+1$ in \cref{lem:forced-cell-pieri},
we obtain
\[
e_{2j}T_{h-j}
\gs e_{2j+1}R_{h-j}
\quad\text{for $1\le j\le a/2-1$}.
\]
Put $\alpha_j=(2j-1)\delta_{a-2j}(b,c)$.
Note that $\delta_{a-2j}(b,c)=\delta_a(b,c)+j$.
Then
\begin{equation}\creflabel[ineq]{pf:s+.Lj}
L_j
\gs (2j-1) e_{2j}B_{h-j}
+(\alpha_j-2j) e_{2j}T_{h-j}.
\end{equation}
If $j\ge 2$,
then $\alpha_j-2j
=(2j-1)\delta_a(b,c)+j(2j-3)\ge 0$ and $L_j\spos$.

It remains to control
\[
L_0+L_1
=-\brk1{B_h+\delta_a(b,c)T_h}
+e_2\brk1{
B_{h-1}
+(\delta_a(b,c)+1)T_{h-1}}
-2e_3 R_{h-1}.
\]
Since~$h-1=(a+b+c)/2-1\ge 2$,
\cref{prop:Pieri} gives
\crefalias{equation}{ineq}
\begingroup
\crefalias{equation}{ineq}
\begin{align}
\creflabel[ineq]{ineq:e2B>B}
e_2 B_{h-1}
&\gs B_h+T_h+s_{h(h-1)11},
\\
\creflabel[ineq]{ineq:e2T>T}
e_2 T_{h-1}
&\gs T_h+s_{h(h-1)11}+\mathcal H_{2h+1},
\quad\text{and}
\\
\notag
e_3 R_{h-1}
&=T_h+s_{h(h-1)11}+\mathcal H_{2h+1}.
\end{align}
\endgroup
\crefalias{equation}{equation}
It follows that
\[
L_0+L_1-\delta_a(b,c)(e_2 T_{h-1}-T_h)
=-B_h
+e_2 B_{h-1}
+e_2 T_{h-1}
-2e_3 R_{h-1}
\gs -\mathcal H_{2h+1}.
\]
If~$\delta_a(b,c)\ge 1$,
then \cref{ineq:e2T>T} gives
\[
L_0+L_1
\gs \delta_a(b,c)\mathcal H_{2h+1}-\mathcal H_{2h+1}
\spos
\quad\text{and}\quad
M_E
=(L_0+L_1)+\sum_{j=2}^{a/2-1}L_j\spos.
\]
Otherwise~$\delta_a(b,c)=0$.
Then $a=5(b+c)+2\ge 12$.
By \cref{pf:s+.Lj,prop:Pieri},
\[
L_2
\gs (\alpha_2-4) e_4 T_{h-2}
=2e_4 T_{h-2}
\gs 2\mathcal H_{2h+1}.
\]
It follows that
\[
L_0+L_1+L_2
\gs -\mathcal H_{2h+1}+L_2
\gs \mathcal H_{2h+1}
\spos
\quad\text{and}\quad
M_E
=(L_0+L_1+L_2)+\sum_{j=3}^{a/2-1}L_j\spos.
\qedhere
\]
\end{proof}

For odd~$a$,
the paired recurrence terms require a different estimate.
The next lemma also identifies the term~$J_h$ needed for the case~$a=3$.

\begin{lemma}\label{lem:odd-block}
Let~$a,b,c$ be positive odd integers with $a=2t+1\ge 5$.
If $\delta_a(b,c)\ge 1$, then $M_O\spos$,
where
\[
M_O
=\sum_{j=0}^{t}(2j-1) e_{2j}(U_{h-j}-R_{h-j})
+\sum_{j=0}^{t-1}
2j\,\delta_{a-2j-1}(b,c) e_{2j+1}T_{h-j-1}.
\]
Moreover,
for every integer~$h\ge 3$ in degree~$N=2h$,
put $J_h=-(U_h-R_h)+e_2(U_{h-1}-R_{h-1})$.
Then $J_h+\mathcal H_{2h}\spos$.
\end{lemma}

\begin{proof}
In degree~$N=2h$ with~$h\ge 3$,
\cref{prop:Pieri} gives the exact identity
\[
J_h
=s_{(h+1)(h-2)1}+s_{h(h-2)1^2}-\mathcal H_{2h}.
\]
Thus $J_h+\mathcal H_{2h}\spos$ throughout this range.
We now prove~$M_O\spos$.
Note that
\[
\delta_{a-2j-1}(b,c)
=\ceil{(5b+5c+2-(a-2j-1))/2}
=\delta_a(b,c)+j.
\]
Reindexing the second sum,
we can write $M_O=J_h+\sum_{j=2}^{t}G_j$,
where
\[
G_j
=(2j-1) e_{2j}(U_{h-j}-R_{h-j})
+2(j-1)\brk1{\delta_a(b,c)+j-1}e_{2j-1}T_{h-j},
\]
with $h-j\ge h-t=(a+b+c+1)/2-(a-1)/2\ge 2$.
By \cref{prop:Pieri},
\[
G_j
\gs (2j-1)(e_{2j-1}T_{h-j}-e_{2j}R_{h-j})
+\beta_j e_{2j-1}T_{h-j},
\]
where
\[
\beta_j
=2(j-1)\brk1{\delta_a(b,c)+j-1}-(2j-1)
\ge 2\cdot 1\cdot j-(2j-1)
=1.
\]
Taking $\lambda=(h-j)^2$ and~$i=2j$ in \cref{lem:forced-cell-pieri},
we obtain $e_{2j-1}T_{h-j}\gs e_{2j}R_{h-j}$.
Thus, for $2\le j\le t$,
\[
G_j
\gs \beta_j e_{2j-1}T_{h-j}\spos.
\]
For $j=2$, we have $\beta_2=2\delta_a(b,c)-1\ge 1$.
By \cref{prop:Pieri}, $e_3 T_{h-2}\gs \mathcal H_{2h}$.
Consequently,
\[
G_2
\gs \beta_2e_3 T_{h-2}
\gs \beta_2\mathcal H_{2h}
\gs \mathcal H_{2h}.
\]
Together with $J_h\gs -\mathcal H_{2h}$ and $G_j\spos$ for $j\ge 3$,
this proves $M_O\spos$.
\end{proof}

The assertion~$M_O\spos$ in \cref{lem:odd-block} does not extend to~$a=3$:
then~$t=1$ and~$M_O=J_h$,
whose coefficient at~$\mathcal H_{2h}$ is~$-1$.
The bound $J_h+\mathcal H_{2h}\spos$ still applies.

We now combine the initial cases with these two block estimates
to prove the theorem by induction.

\begin{proof}[Proof of \cref{thm:threshold-engine}]
We use induction on $N=a+b+c+1$,
simultaneously for all triples satisfying the hypotheses.
For odd $a$, put $N=a+3=2h$.
By \cref{prop:3spider},
\[
f_a(1,1)
=\bigl(X_N-e_2 X_{N-2}\bigr)
+\bigl(\Delta_N+s_{hh}-s_{(h+1)(h-1)}\bigr).
\]
The first summand is $e$-positive by \cref{thm:TW+},
and the second is Schur positive by \cref{thm:path-bridge}.
Hence $f_a(1,1)\spos$.
On the other hand, recall that $C_2(b,c)=X_{S(2,b,c)}$.
By \cref{prop:universal-core,prop:critical},
\[
f_2(b,c)
=X_{S(2,b,c)}
-s_{(h+1)h}
-\delta_2(b,c) s_{hh1}
\spos.
\]
These base families include every admissible triple with $N\le 5$.
For the remaining triples,
assume the result for all admissible triples with fewer than $N$ vertices.

Suppose first that $a\ge 4$ is even.
For $1\le j\le a/2-1$, we have $\delta_{a-2j}(b,c)=\delta_a(b,c)+j\ge 0$.
Thus every smaller even-leg spider in the recurrence
satisfies the hypotheses of the theorem.
Applying the induction hypothesis to the smaller spiders
and using \cref{prop:rec.master,prop:universal-core,lem:even-block},
we obtain
\begin{align*}
f_a(b,c)
&=C_a(b,c)
+\sum_{j=1}^{a/2-1}
\brk2{
(2j-1) e_{2j}X_{S(a-2j,\,b,\,c)}
+2j\,e_{2j+1}X_{S(a-2j-1,\,b,\,c)}
}
-B_h-\delta_a(b,c)T_h
\\
&\gs \sum_{j=1}^{a/2-1}\brk2{
(2j-1) e_{2j}
\brk1{B_{h-j}+\delta_{a-2j}(b,c)T_{h-j}}
-2j\,e_{2j+1}R_{h-j}
}
+L_0
=M_E
\spos.
\end{align*}

It remains to consider odd $a$.
In this case,
$f_a(b,c)$ depends only on the multiset of the three leg lengths.
The base family $f_a(1,1)$ covers all triples with two legs of length $1$.
For every remaining triple,
we may therefore permute the legs so that $c\le a\le b$ and $a\ge 3$.
Then
\[
\delta_a(b,c)
=(5b+5c+3-a)/2
\ge 2a+(5c+3)/2
\ge 10.
\]
First suppose~$a=3$.
Here $N=b+c+4=2h$ and $h\ge 3$,
so \cref{lem:odd-block} gives~$J_h\gs -\mathcal H_{2h}$.
By \cref{prop:rec.master},
\begin{align*}
f_3(b,c)
&=C_3(b,c)+e_2 X_{S(1,b,c)}-U_h+R_h
=C_3(b,c)+e_2 f_1(b,c)+J_h
\\
&\gs \brk1{C_3(b,c)-\mathcal H_{2h}}
+e_2 f_1(b,c)
\spos
\end{align*}
by \cref{lem:s+:C3} and by induction.

Now suppose that $a=2t+1\ge 5$.
For $1\le j\le t-1$, we have $\delta_{a-2j-1}(b,c)=\delta_a(b,c)+j\ge 1$.
Hence the smaller even-leg spiders satisfy the hypotheses of the theorem.
The induction hypothesis gives
\begin{align*}
X_{S(a-2j,\,b,\,c)}
&\gs U_{h-j}-R_{h-j}
\quad\text{for $1\le j\le t$, and}
\\
X_{S(a-2j-1,\,b,\,c)}
&\gs B_{h-j-1}
+\delta_{a-2j-1}(b,c)T_{h-j-1},
\quad\text{for $1\le j\le t-1$}.
\end{align*}
By \cref{prop:rec.master,prop:universal-core,lem:odd-block},
we can then deduce that
\begin{align*}
f_a(b,c)
&=C_a(b,c)-U_h+R_h
+\sum_{j=1}^{t}(2j-1) e_{2j}X_{S(a-2j,\,b,\,c)}
+\sum_{j=1}^{t-1}2j\,e_{2j+1}X_{S(a-2j-1,\,b,\,c)}
\\
&\gs -U_h+R_h
+\sum_{j=1}^{t}(2j-1) e_{2j}(U_{h-j}-R_{h-j})
+\sum_{j=1}^{t-1}2j\,e_{2j+1}
\bigl(B_{h-j-1}+\delta_{a-2j-1}(b,c)T_{h-j-1}\bigr)
\\
&\gs M_O\spos.
\end{align*}
This completes the induction proof.
\end{proof}

\section{Even-leg bounds and completion of the classification}\label{sec:proof-two-even}

Throughout this section,
we fix an even integer~$a\ge 2$,
and denote~$X_{b,c}=X_{S(a,b,c)}$.
Recall that
\[
N
=a+b+c+1,\quad
h
=\floor{N/2},
\quad\text{and}\quad
\delta_a(b,c)
=\ceil{(5b+5c+2-a)/2}.
\]
We avoid confusion by writing~$h_{b,c}$ for~$h$
when the graph is unclear from context.
By definition,
the following identities hold whenever all shifted leg lengths remain positive:
\begin{align}
&\label{h.-2}
h_{b,\,c-2}
=h_{b-2,\,c}
=h_{b,c}-1,
\\
&\label{h.-1}
h_{b-1,\,c}
=h_{b,\,c-1}
=h_{b,c}-\mathbf 1(\text{$b+c$ is odd}),
\\
&\label{delta.-2}
\delta_a(b-2,\,c)
=\delta_a(b,\,c-2)
=\delta_a(b,c)-5,
\quad\text{and}
\\
&\label{delta.-1}
\delta_a(b-1,\,c)
=\delta_a(b,\,c-1)
=\delta_a(b,c)-2-\mathbf 1(\text{$b+c$ is odd}).
\end{align}
We use $c$ and $p$ to denote positive odd integers,
and use $q$ to denote a positive even integer.

In this section,
we first prove Schur positivity for spiders with exactly two even legs,
including~$S(a,q,1)$.
Then \cref{cor:propagation} gives Schur positivity of $S(a,q,r)$
for every integer $r\ge 1$,
including the case of three even legs.
Define
\[
g_{b,c}
=\begin{dcases*}
X_{b,c}
-s_{(h+1)h}
-\delta_a(b,c) s_{hh1},
& if $b+c$ is even,
\\
X_{b,c}
+2\delta_a(b-1,\,c) e_1 s_{(h-1)^21},
& if $b+c$ is odd.
\end{dcases*}
\]
Then $g_{b,c}=g_{c,b}$ by \cref{delta.-1}.

For odd $b,c$ with $\delta_a(b,c)\ge 0$,
the bound $g_{b,c}\spos$ is exactly the even-$a$ assertion
of \cref{thm:threshold-engine}.
The following proposition extends this bound to $\delta_a(b,c)<0$
and simultaneously proves Schur positivity when $b$ is even.
In the latter case,
the correction term allows adjacent terms in the master recurrence
to compensate for a negative coefficient of a neighboring odd-leg spider.

\begin{proposition}\label{prop:general-even-bounds}
Let~$a\ge 2$ be even.
For any integer~$b\ge 1$ and any odd~$c\ge 1$,
we have $g_{b,c}\spos$.
Moreover, if $b$ is even,
then $X_{b,c}\spos$.
\end{proposition}

We first establish the initial cases and recurrence estimates,
then prove both assertions by simultaneous induction
near the end of the section.
\cref{lem:s+:Y1:Z2} supplies the initial cases
for the induction on the two varying leg lengths.
We will need the following identity given by \cref{prop:Pieri}:
\begin{equation}\label{e1Th}
e_1 T_h
=s_{(h+1)h1}+s_{hh2}+s_{hh11}.
\end{equation}

\begin{lemma}\label{lem:s+:Y1:Z2}
We have~$g_{1,1}\spos$,
$X_{2,c}\spos$,
and~$g_{2,c}\spos$ for any odd integer $c\ge 1$.
\end{lemma}

\begin{proof}
Write~$a=2k$ and put $f_k=X_{1,1}-\Delta_{2k+3}$.
By \cref{prop:3spider,thm:TW+},
\[
f_k
=(X_{2k+3}+e_1 X_{2k+2}-X_2 X_{2k+1})
-(e_1 X_{2k+2}-e_2 X_{2k+1})
=X_{2k+3}-e_2 X_{2k+1}\epos.
\]
Note that $\delta_{2k}(1,1)=\ceil{(12-2k)/2}=6-k$.
By \cref{prop:critical,thm:path-bridge},
\[
\Coeff{T_{k+1}}{f_k}
=\delta_{2k}(1,1)-\brk1{3-(k+1)}
=4.
\]
On the other hand, by \cref{thm:path-bridge},
$\Delta_{2k+3}+(k-2)T_{k+1}-B_{k+1}\spos$.
As a result,
\[
g_{1,1}
=X_{1,1}
-B_{k+1}
-\delta_{2k}(1,1)T_{k+1}
=(f_k-4T_{k+1})
+(\Delta_{2k+3}+(k-2)T_{k+1}-B_{k+1})\spos.
\]

Consider~$X_{2,c}$.
Then
\[
h
=(a+c+3)/2
\ge 3
\quad\text{and}\quad
\delta_a(1,c)
=(5c+7-a)/2
=3c+5-h.
\]
By \cref{lem:core-recurrence}, for $t\ge 2$, 
$C_t(a,c)
\gs e_2^{h-2}C_t(0,1)$.
Since~$\mathcal C_2(0,1)\cong P_4$,
taking~$t=2$ gives
\[
X_{2,c}
=C_2(a,c)
\gs e_2^{h-2}X_4
\gs 2s_{22}e_2^{h-2}\spos.
\]
If~$\delta_a(1,c)\ge 0$, then $g_{2,c}
=X_{2,c}+2\delta_a(1,c) e_1 T_{h-1}\spos$.
It remains to consider~$\delta_a(1,c)\le -1$.

In view of \cref{e1Th},
we consider three Schur coefficients in $e_2^{h-2}s_{22}$:
\[
\Coeff{s_{h(h-1)1}}{e_2^{h-2}s_{22}}
=h-2,
\qquad
\Coeff{s_{(h-1)^22}}{e_2^{h-2}s_{22}}
=\binom{h-2}{2},
\qquad
\Coeff{s_{(h-1)^21^2}}{e_2^{h-2}s_{22}}
=\binom{h-1}{2}.
\]
Since~$\delta_a(1,c)\le -1$,
we find~$h\ge 9$ and all three coefficients above are at least~$h-2$.
Therefore,
\begin{align*}
g_{2,c}
&=X_{2,c}+2\delta_a(1,c) e_1 T_{h-1}
\gs 2 s_{22}e_2^{h-2}-2(h-3c-5) e_1 T_{h-1}
\\
&\gs 2(h-2) e_1 T_{h-1}-2(h-3c-5) e_1 T_{h-1}
=6(c+1) e_1 T_{h-1}
\spos.
\qedhere
\end{align*}
\end{proof}

The even-leg induction will also require a comparison between products
involving~$T_{h-1}$ and~$T_h$.
We establish it directly from \cref{prop:Pieri}.

\begin{lemma}\label{lem:LR-difference}
For any integer~$h\ge 3$,
we have~$s_{21}T_{h-1}\gs e_1 T_h$.
\end{lemma}

\begin{proof}
Since~$s_{21}=e_{21}-e_3$,
we compute $[s_\lambda]e_{21}T_{h-1}$ and $[s_\lambda]e_3 T_{h-1}$
for each of the three shapes $\lambda\in\{(h+1)h1,\,hh2,\,hh11\}$,
whose difference gives $[s_\lambda]s_{21}T_{h-1}$.
Two applications of \cref{prop:Pieri} give the three pairs of coefficients
in \cref{tab:LR-difference-pieri}.
\begin{table}[htbp]
\centering
\caption{Coefficients used to compare~$s_{21}T_{h-1}$ and~$e_1 T_h$ for~$h\ge 3$.}
\label{tab:LR-difference-pieri}
\setlength{\tabcolsep}{12pt}
\renewcommand{\arraystretch}{1.2}
\begin{tabular}{*{4}{>{$}c<{$}}}
\toprule
\lambda & \Coeff{s_\lambda}{e_{21}T_{h-1}} & \Coeff{s_\lambda}{e_3 T_{h-1}} & \Coeff{s_\lambda}{s_{21}T_{h-1}}
\\
\midrule
(h+1,\,h,\,1) & 1 & 0 & 1
\\
(h,h,2) & 2 & 1 & 1
\\
(h,h,1,1) & 2 & 1 & 1
\\
\bottomrule
\end{tabular}
\end{table}
In view of the last column and \cref{e1Th},
each constituent of~$e_1 T_h$ therefore occurs with coefficient~$1$ in~$s_{21}T_{h-1}$.
This completes the proof.
\end{proof}

In view of \cref{prop:rec.master},
we consider the sum
\[
L_{i,q,c}
=(i-1) e_i X_{q,c}+i e_{i+1}X_{q-1,\,c}
\]
of two adjacent terms.
Let $h_{q,c}=(a+q+c+1)/2$.
By \cref{h.-1},
\begin{align}
\notag
L_{i,q,c}
&=(i-1) e_i\brk1{g_{q,c}-2\delta_a(q-1,\,c) e_1 T_{h_{q,c}-1}}
+i e_{i+1}(g_{q-1,\,c}
+B_{h_{q-1,c}}
+\delta_a(q-1,\,c)T_{h_{q-1,c}})
\\
\label{eq:block-decomposition}
&=(i-1) e_i g_{q,c}+i e_{i+1}g_{q-1,\,c}
+i e_{i+1}B_{h_{q-1,\,c}}
-\delta_a(q-1,\,c)\brk1{2(i-1) e_{i1}-i e_{i+1}}T_{h_{q-1,\,c}}.
\end{align}

This decomposition gives Schur positivity of a pair of recurrence terms
under the hypotheses available in the induction.

\begin{lemma}\label{lem:adjacent-block}
Let~$q\ge 2$ be even and~$c\ge 1$ be odd.
Suppose that
\begin{enumerate}
\item
if $\delta_a(q-1,\,c)\ge 0$, then $X_{q,c}\spos$,
and
\item
if $\delta_a(q-1,\,c)<0$,
then $g_{q,c}\spos$ and $g_{q-1,\,c}\spos$.
\end{enumerate}
Then~$L_{i,q,c}\spos$ for any~$i\ge 2$.
\end{lemma}

\begin{proof}
If $\delta_a(q-1,\,c)\ge 0$, then $X_{q,c}\spos$ by hypothesis,
and~$X_{q-1,\,c}\spos$ by \cref{thm:threshold-engine}.
Thus $L_{i,q,c}\spos$ by definition.
If $\delta_a(q-1,\,c)<0$,
then \cref{prop:Pieri} implies
\[
2(i-1) e_{i1}-i e_{i+1}
=2(i-1) s_{21^{i-1}}+(i-2) e_{i+1}\spos.
\]
Hence \cref{eq:block-decomposition} gives $L_{i,q,c}\spos$.
\end{proof}

Put $D_{p,c}=X_{p+2,\,c}-e_2 X_{p,c}$.
Define $E_h=e_2 T_h-T_{h+1}$.
Then $E_h\spos$ by \cref{ineq:e2T>T} with~$h$ replaced by~$h+1$.
The next estimate shows that~$D_{p,c}$ contains enough Schur-positive terms
to compensate for the change
in the correction defining~$g_{p,c}$.

\begin{lemma}\label{lem:odd-payment}
Let~$c,p\ge 1$ be odd.
If~$p\ge 3$,
assume additionally that
\[
\delta_a(p-2,\,c)
\le 0,
\qquad g_{p-1,\,c}\spos,
\qquad g_{p-2,\,c}\spos,
\qquad D_{p,c}
\gs L_{3,\,p-1,\,c}.
\]
Put~$h=h_{p,c}$.
Then
\[
D_{p,c}
\gs 4T_{h+1}+
\begin{dcases*}
4(h-1)E_h, & if $p=1$,
\\
-4\delta_a(p-2,\,c)E_h, & if $p\ge 3$.
\end{dcases*}
\]
Consequently, $D_{p,c}\gs 4T_{h+1}+(20-4\delta_a(p,c))E_h$.
\end{lemma}

\begin{proof}
Let~$h=h_{p,c}$.
Suppose that~$p=1$.
Then $h=(a+c+1)/2\ge 2$.
It is direct to compute that
\[
C_3(0,1)
=5e_5+3e_{41}+4e_{32}
\gs 4s_{221}.
\]
By \cref{prop:rec.master,lem:core-recurrence},
\[
D_{1,c}
=X_{3,c}-e_2 X_{1,c}
=C_3(a,c)
\gs e_2^{h-1}C_3(0,1)
\gs 4e_2^{h-1}s_{221}.
\]
We claim that each Schur constituent of
\[
E_h
=e_2 T_h-T_{h+1}
=s_{(h+1)h2}
+s_{(h+1)h11}
+s_{hh21}
+s_{hh111}
\]
occurs at least~$h-1$ times in the Schur expansion of~$e_2^{h-1}s_{221}$.
In fact,
start with~$(2,2,1)$ and normally add to rows~$1,2$;
at one chosen step add instead to rows~$1,3$;
$1,4$;
$3,4$;
or~$4,5$, respectively.
Adding to rows~$1,2$ at every step gives~$T_{h+1}$.
This proves the claim.
Consequently,
\[
D_{1,c}
\gs 4e_2^{h-1}s_{221}
\gs 4T_{h+1}+4(h-1)E_h.
\]

Now let~$p\ge 3$.
The assumption and \cref{delta.-2} give
\[
5-\delta_a(p,c)
=-\delta_a(p-2,\,c)
\ge 0.
\]
Note that $4e_{31}-3e_4=4s_{211}+e_4$.
By the hypotheses and \cref{eq:block-decomposition,h.-2,delta.-2},
\begin{align*}
D_{p,c}
\gs L_{3,\,p-1,\,c}
&=2e_3 g_{p-1,\,c}
+3e_4 g_{p-2,\,c}
+3e_4 B_{h-1}
-\delta_a(p-2,\,c)(4e_{31}-3e_4)T_{h-1}
\\
&\gs 2e_3 g_{p-1,\,c}
-4\delta_a(p-2,\,c) s_{211}T_{h-1}.
\end{align*}

The bipartite graph~$S(p-1,\,a,\,c)$ has both color classes
of size~$h_{p-1,\,c}=h$ by \cref{h.-1}.
Thus
\[
X_{p-1,\,c}^{[2]}
=2(x_1x_2)^h
=2s_{hh}(x_1,x_2),
\]
which implies $\Coeff{R_h}{X_{p-1,\,c}}=2$.
It follows that
\[
\Coeff{R_h}{g_{p-1,\,c}}
=\Coeff{R_h}{\brk1{X_{p-1,\,c}
+2\delta_a(p-2,\,c) e_1 T_{h-1}}}
=\Coeff{R_h}{X_{p-1,\,c}}
=2.
\]
Since $g_{p-1,\,c}\spos$ by hypothesis, we find $g_{p-1,\,c}\gs 2R_h$.
Thus by \cref{prop:Pieri},
\[
2e_3 g_{p-1,\,c}
\gs 4e_3 R_h
\gs 4T_{h+1}.
\]
Two applications of \cref{prop:Pieri} give the coefficients
in the first two columns in \cref{tab:odd-payment-pieri}.
The last column is their difference since $s_{211}=e_{31}-e_4$.
Since every entry in the last column is at least $1$,
we obtain $s_{211}T_{h-1}\gs E_h$.
Since $\delta_a(p-2,\,c)\le 0$,
combining these bounds gives
\[
D_{p,c}
\gs 4T_{h+1}-4\delta_a(p-2,\,c) s_{211}T_{h-1}
\gs 4T_{h+1}-4\delta_a(p-2,\,c)E_h.
\]

\begin{table}[htbp]
\centering
\caption{Coefficients used to bound~$s_{211}T_{h-1}$ for~$h\ge 3$.}
\label{tab:odd-payment-pieri}
\setlength{\tabcolsep}{12pt}
\renewcommand{\arraystretch}{1.2}
\begin{tabular}{>{$}l<{$}*{3}{>{$}c<{$}}}
\toprule
\lambda & \Coeff{s_\lambda}{e_{31}T_{h-1}} & \Coeff{s_\lambda}{e_4 T_{h-1}} & \Coeff{s_\lambda}{s_{211}T_{h-1}}
\\
\midrule
(h+1,\,h,\,2) & 1 & 0 & 1
\\
(h+1,\,h,\,1,\,1) & 1 & 0 & 1
\\
(h,h,2,1) & 3 & 1 & 2
\\
(h,h,1,1,1) & 2 & 1 & 1
\\
\bottomrule
\end{tabular}
\end{table}

For the uniform bound,
observe that
\begin{align*}
4(h-1)-(20-4\delta_a(1,c))
=12c-8
\ge 4,
&\quad\text{when~$p=1$, whereas}
\\
-4\delta_a(p-2,\,c)
=20-4\delta_a(p,c),
&\quad\text{when~$p\ge 3$}.
\end{align*}
Since~$E_h\spos$,
these comparisons give the claimed uniform bound.
\end{proof}

We now use the initial cases and the preceding estimates
to prove both assertions of \cref{prop:general-even-bounds} simultaneously.

\begin{proof}[Proof of \cref{prop:general-even-bounds}]
We proceed by induction on $b+c$,
simultaneously for all integers~$b\ge 1$ and odd~$c\ge 1$.
The base cases $(b,c)=(1,1)$ and~$b=2$ are supplied by \cref{lem:s+:Y1:Z2}.

First suppose that~$b$ is odd.
If $\delta_a(b,c)\ge 0$,
then $g_{b,c}=f_a(b,c)\spos$ by \cref{thm:threshold-engine}.
We may therefore assume that $\delta_a(b,c)<0$.
Outside the base case $(1,1)$,
at least one of~$b$ and $c$ is at least~$3$.
By the symmetry~$g_{b,c}=g_{c,b}$,
we may therefore assume that~$b\ge 3$.
By \cref{prop:rec.master},
\begin{align*}
D_{b-2,\,c}-C_b(a,c)
&=X_{b,c}-e_2 X_{b-2,\,c}-C_b(a,c)
=\sum_{i=3}^{b-1}(i-1) e_i X_{b-i,\,c}
\\
&=\sum_{3\le i\le b-2,\ \text{$i$ odd}}
\bigl((i-1) e_i X_{b-i,\,c}
+i e_{i+1}X_{b-i-1,\,c}\bigr)
=\sum_{3\le i\le b-2,\ \text{$i$ odd}}
L_{i,\,b-i,\,c}.
\end{align*}
By \cref{prop:universal-core}, we know $C_b(a,c)\spos$.
By induction and by \cref{lem:adjacent-block},
we have $L_{i,\,b-i,\,c}\spos$ for each $i$.
In order to apply \cref{lem:odd-payment} for the pair $(b-2,\,c)$,
we check its conditions.
Then $\delta_a(b-2,\,c)=\delta_a(b,c)-5\le -6$.
When $b\ge 5$,
induction gives $g_{b-3,\,c}\spos$ and $g_{b-4,\,c}\spos$.
Moreover,
\[
\delta_a(b-4,\,c)
=\delta_a(b,c)-10
\le -11.
\]
The sum contains $L_{3,\,b-3,\,c}$,
so $D_{b-2,c}\gs L_{3,b-3,c}$.
Thus all hypotheses of \cref{lem:odd-payment} hold for $p=b-2$.
For $b=3$,
its $p=1$ case applies directly.
Let~$h=h_{b-2,\,c}$.
By \cref{ineq:e2B>B}, $e_2 B_h-B_{h+1}\gs T_{h+1}$.
Applying \cref{lem:odd-payment} with $p=b-2$,
and writing $d=\delta_a(b-2,c)\le -6$,
we obtain
\begin{align*}
g_{b,c}-e_2 g_{b-2,\,c}
&=X_{b,c}
-B_{h_{b,c}}
-\delta_a(b,c)T_{h_{b,c}}
-e_2\brk1{
X_{b-2,\,c}
-B_{h_{b-2,\,c}}
-d T_{h_{b-2,\,c}}
}
\\
&=D_{b-2,\,c}
+e_2 B_h
-B_{h+1}
-(d+5)T_{h+1}
+d e_2 T_h
\\
&\gs D_{b-2,\,c}
+d E_h
-4T_{h+1}
\gs (20-3d)E_h\spos.
\end{align*}
Since~$g_{b-2,\,c}\spos$ by induction, we obtain~$g_{b,c}\spos$.

It remains to consider even~$b\ge 4$.
By \cref{prop:rec.master},
\[
X_{b,c}-C_b(a,c)
=\sum_{2\le i\le b-2,\text{ $i$ even}}
\bigl((i-1) e_i X_{b-i,\,c}
+i e_{i+1}X_{b-i-1,\,c}\bigr)
=\sum_{2\le i\le b-2,\text{ $i$ even}}
L_{i,\,b-i,\,c}.
\]
By \cref{prop:universal-core,lem:adjacent-block} and induction,
we find
\[
X_{b,c}\spos
\quad\text{and}\quad
X_{b,c}
\gs L_{2,\,b-2,\,c}.
\]
If $\delta_a(b-1,\,c)\ge 0$,
then
\[
g_{b,c}
=X_{b,c}+2\delta_a(b-1,\,c) e_1 T_{h_{b,c}-1}
\spos,
\]
which completes this case.
Otherwise $\delta_a(b-1,\,c)<0$;
put~$h=h_{b,c}$.
By \cref{eq:block-decomposition,h.-1,h.-2,delta.-2},
\begin{align*}
L_{2,\,b-2,\,c}
&=e_2 g_{b-2,c}
+2e_3 g_{b-3,\,c}
+2e_3 B_{h_{b-3,\,c}}
-\delta_a(b-3,\,c)
(2e_{21}-2e_3)T_{h_{b-3,\,c}}
\\
&=e_2 g_{b-2,c}
+2e_3 g_{b-3,\,c}
+2e_3 B_{h-2}
-2(\delta_a(b-1,\,c)-5)
s_{21}T_{h-2}.
\end{align*}
Thus by induction and by \cref{lem:LR-difference},
\begin{align*}
g_{b,c}
&=X_{b,c}+2\delta_a(b-1,\,c) e_1 T_{h-1}
\gs L_{2,\,b-2,\,c}+2\delta_a(b-1,\,c) e_1 T_{h-1}
\\
&\gs -2\delta_a(b-1,\,c)(s_{21}T_{h-2}-e_1 T_{h-1})\spos.
\end{align*}
This completes the induction.
\end{proof}

We conclude by proving \cref{thm:classification,thm:refined-bounds}.

\begin{proof}[Proof of \cref{thm:classification}]
Let~$G=S(a,b,c)$.
We prove the three assertions in order.
\begin{enumerate}[label=\textup{(\roman*)}]
\item
If~$a,b,c$ are all odd,
then~$G$ is not Schur positive by \cref{thm:s+.bipartite}.

\item
Suppose~$a$ is even and~$b,c$ are odd.
Put $h=(a+b+c)/2$.
By \cref{prop:critical},
\[
\Coeff{s_{hh1}}{X_G}
=\delta_a(b,c)
=(5b+5c+2-a)/2.
\]
Thus Schur positivity requires $a\le 5b+5c+2$.
Conversely, if $a\le 5b+5c+2$,
then~$\delta_a(b,c)\ge 0$.
By \cref{thm:threshold-engine},
\[
X_G-s_{(h+1)h}-\delta_a(b,c) s_{hh1}\spos.
\]
Adding the two Schur-positive terms back gives~$X_G\spos$.
This proves the equivalence.

\item
If at least two legs are even,
we may suppose that~$a,b$ are even.
By \cref{prop:general-even-bounds}, $X_{b,1}=X_{S(a,b,1)}\spos$.
Then \cref{cor:propagation} gives~$G=S(a,b,c)\spos$ for every~$c\ge 1$.
\qedhere
\end{enumerate}
\end{proof}

\begin{proof}[Proof of \cref{thm:refined-bounds}]
Let $G=S(a,b,c)$.
When $a$ is odd,
the two bipartition classes of $G$ have sizes $h-1$ and $h+1$.
Since $G$ is connected,
its proper colorings with colors $1$ and $2$ give
\[
X_G^{[2]}
=(x_1x_2)^{h-1}(x_1^2+x_2^2)
=s_{(h+1)(h-1)}^{[2]}-s_{hh}^{[2]}.
\]
The linear independence of the two-variable Schur polynomials
therefore gives $\Coeff{s_{hh}}{X_G}=-1$.
This is also proved in the proof of \cref{thm:s+.bipartite} by \citet{WW20}.
By \cref{thm:threshold-engine},
every other Schur coefficient of $X_G$ is nonnegative.

When $a$ is even, \cref{prop:general-even-bounds} gives~$g_{b,c}\spos$,
that is, $X_G-B_h-\delta_a(b,c)T_h\spos$.
By \cref{prop:critical},
$[T_h]X_G=\delta_a(b,c)$ is the only possible negative Schur coefficient.
Since $G$ fails Schur positivity by hypothesis,
this coefficient is strictly negative.
\end{proof}

\section{Disclosure of artificial intelligence use}\label{sec:ai-use}

We used ChatGPT~6 Pro through its web interface and OpenAI Codex.

AI assistance in developing proof strategies was essential
to completing the proof of \cref{thm:classification}.
The authors formulated the classification problem and guided the search
for a proof.
Despite the substantial effort devoted to establishing
the Schur-positivity classification results for clique-spiders
in \citet{WW26X},
we did not expect a complete classification of Schur-positive
$3$-legged spiders to be within reach.
To our surprise,
AI proposed combining induction with the $3$-clique-spider results
recalled in \cref{thm:s+:clique-spiders} to obtain such a classification.

AI supplied the overall strategy of simultaneous induction
and was the first to resolve the main obstacles to proving the classification.
The approach strengthened the desired positivity conclusions
by introducing correction terms,
so that the recurrences allowed positive contributions in one parity case
to offset negative contributions in another.
The authors studied this approach and verified that it established
the claimed results.
They then guided AI in deriving \cref{thm:refined-bounds}.

The authors subsequently proposed simpler proofs of several local lemmas
using combinatorial interpretations of Schur functions
and applications of Pieri's rule.
They also made numerous suggestions for restating and reorganizing
the initially unclear and scattered lemmas
into a coherent argument that human readers could follow.
They further simplified the overall proof strategy.
In particular,
the authors introduced the function~$g_{b,c}$ in \cref{prop:general-even-bounds}
to unify the auxiliary functions that AI had initially proposed
for separate cases.

OpenAI Codex assisted with proof revisions,
computational checks,
and literature searches.
It substantially improved the efficiency of writing and editing,
especially the preparation and revision of tables and the \LaTeX{} source.

The authors have independently verified every result stated in this paper
and take full responsibility for its accuracy and content.

\bibliographystyle{plainnat}
\IfFileExists{../csf.bib}{\bibliography{../csf}}{\bibliography{csf}}

\end{document}